\documentclass[12pt]{article} 
\usepackage[utf8]{inputenc}
\usepackage[english]{babel}
\usepackage{amsfonts}
\usepackage{amstext}
\usepackage{graphicx}
\usepackage{amsmath}
\usepackage{amssymb}
\usepackage{euscript}
\usepackage{amsthm}
\usepackage{enumerate}

\numberwithin{equation}{section} 

\renewcommand{\Re}{\mathop{\rm Re}\nolimits}
\renewcommand{\Im}{\mathop{\rm Im}\nolimits}

\graphicspath{{Graphics/}}
\newtheorem{theorem}{Theorem}
\newtheorem{proposition}{Proposition}
\newtheorem{remark}{Remark}
\newtheorem{corollary}{Corollary}

\begin{document}
\title{Long-time asymptotics for the integrable nonlocal focusing  nonlinear Schr\"odinger equation for a family of step-like initial data}
\author{Ya. Rybalko$^{\dag}$ and D. Shepelsky$^{\dag,\ddag}$\\
 \small\em {}$^\dag$ B.Verkin Institute for Low Temperature Physics and Engineering\\ \small\em {} of the National Academy of Sciences of Ukraine\\
 \small\em {}$^\ddag$ V.Karazin Kharkiv National University}

\date{}

\maketitle

\begin{abstract}
We study the Cauchy problem for 
the integrable  nonlocal focusing nonlinear Schr\"odinger (NNLS) equation
$
iq_{t}(x,t)+q_{xx}(x,t)+2 q^{2}(x,t)\bar{q}(-x,t)=0
$
with the  step-like initial data close to the ``shifted step function'' 
$\chi_R(x)=AH(x-R)$, where 
$H(x)$ is the Heaviside step function, and $A>0$ and $R>0$ are arbitrary constants. 
Our main aim is to study the large-$t$ behavior of the solution of this problem.  
We show that for 
$R\in\left(\frac{(2n-1)\pi}{2A},\frac{(2n+1)\pi}{2A}\right)$, $n=1,2,\dots$,
the $(x,t)$ plane splits into 
 $4n+2$ sectors exhibiting different  asymptotic behavior. Namely, there are
 $2n+1$ sectors where the solution decays to $0$, whereas in the other $2n+1$ sectors
(alternating with the sectors with decay), the solution approaches  (different) constants along each ray $x/t=const$. 
Our main technical tool is the representation
of the solution of the Cauchy problem in terms of the solution of an associated matrix Riemann-Hilbert problem and its subsequent  asymptotic analysis following the ideas of
nonlinear steepest descent method.
\end{abstract}

\section{Introduction}
We consider the  Cauchy problem for the integrable nonlocal focusing 
nonlinear Schr\"odinger (NNLS) equation with  step-like initial data:
\begin{subequations}
\label{1}
\begin{align}
\label{1-a}
& iq_{t}(x,t)+q_{xx}(x,t)+2q^{2}(x,t)\bar{q}(-x,t)=0,\qquad x\in\mathbb{R},\,t>0,\\
\label{1-b}
& q(x,0)=q_0(x), \qquad x\in\mathbb{R},
\end{align}
where
\begin{equation}
\label{shifted-step-gen}
q_0(x)\to
\begin{cases}
0,\quad x\to -\infty,\\
A,\quad x\to \infty,
\end{cases}
\end{equation}
\end{subequations}
with some $A>0$.
Here and below $\bar{q}$ denotes the complex conjugate of $q$.
We assume also that the solution $q(x,t)$ satisfies the  boundary conditions 
(\ref{shifted-step-gen}) for all 
$t>0$:
\begin{equation}
\label{2}
q(x,t)\to
\begin{cases}
0,\quad x\to -\infty,\\
A,\quad x\to \infty,
\end{cases}
\end{equation}
where the convergence is sufficiently fast.

The NNLS equation (\ref{1-a}) was introduced by Ablowitz and Musslimani in \cite{AMP} 
as a  reduction $r(x,t)=-\bar{q}(-x,t)$ in
the coupled system of nonlinear Schr\"odinger equations:
\begin{subequations}\label{ss}
\begin{align}
& iq_t(x,t)+q_{xx}(x,t)-2q^2(x,t)r(x,t)=0,\\
& ir_t(x,t)-r_{xx}(x,t)+2r^2(x,t)q(x,t)=0
\end{align}
\end{subequations}
(see also \cite{F16} for multidimensional versions of the NNLS)
and has attracted much attention in recent years due to its distinctive properties. Particularly, the NNLS equation is  $PT$ symmetric \cite{BB}, i.e., if $q(x,t)$ is its solution, so is $\bar{q}(-x,-t)$. 
Therefore, the NNLS equation is closely related to the $PT$ symmetric theory, which is a
field  in modern physics being actively studied  (see e.g. \cite{BKM, B, ZB, KYZ, GA} and references therein). Also this equation has unusual properties of the exact soliton and breather solutions, particularly, solitons could blow up in finite time and the NNLS equation supports both dark and anti-dark soliton solutions simultaneously (see e.g. \cite{AMFL, AMN, SMMC, GP, MS, V, Y} and references therein; see also \cite{Ybook}, where the general soliton solutions for the coupled 
Schr\"odinger equations (\ref{ss}) are found).

Apart from deriving  exact solutions of the NNLS equation, it is  important, in the both mathematical and physical perspective, to consider initial value problems with general initial data. The NNLS equation is an integrable system, i.e. it is a compatibility condition of the two linear equations, the so-called Lax pair (see (\ref{LP}) below) and therefore it can be, in principle, treated by the powerful Inverse Scattering Transform (IST) method \cite{AS, FT, NMPZ}. This method allows  reducing the original nonlinear problem to a sequence of linear ones, and in this sense to find the ``exact'' representation of the solution. The IST method was successfully applied in \cite{AMN} to the Cauchy problem for the NNLS equation on the whole line in the class of functions rapidly decaying as $x\to\pm\infty$  (see also \cite{GS}, where the complete integrability of (\ref{1-a}) in this class was proved).

Although the IST method provides some sort of exact formulas for the solutions, the qualitative analysis of the Cauchy problem for the NNLS equation, particularly, the long-time asymptotics of its solution, is a challenging problem. By using the IST method, the original problem for an integrable system can be reduced to the matrix Riemann-Hilbert (RH) factorization problem  in the complex plane of the spectral parameter. The jump matrix in this problem is oscillatory, which allows applying  the so-called  nonlinear steepest decent method \cite{DZ} for studying its long-time behavior. 
This method was inspired by earlier works of Manakov \cite{M} and Its \cite{I1} and finally put in the rigorous and systematic shape by Deift and Zhou \cite{DZ} (see also \cite{DIZ, DVZ94, DVZ97, BV07, DKKZ, MM} and references therein concerning the Deift and Zhou method and its extensions). The nonlinear steepest decent method consists in  series of transformations of the original RH problem in such a way that for the large values of a parameter (say, 
 the time $t$ in the original nonlinear evolutionary equation), this problem can be solved explicitly, in terms of the special functions (e.g., the parabolic cylinder functions, Riemann theta functions, Painleve transcendents etc.).


The step-like boundary values were considered for the variety of integrable systems, which include the Korteweg-de Vries equation \cite{GP73, AE1, EG, H}, the focusing and defocusing NLS equations \cite{BKS, BFP}, the Toda lattice \cite{DKKZ, VDO}, the modified Korteweg-de Vries equation \cite{KM} among many others. For such conditions a wide range of important physical phenomena of the solutions for the large $t$ are manifested, e.g. collisionless and dispersive shock waves \cite{GP}, rarefaction waves \cite{J}, asymptotic solitons \cite{KK}, modulated waves \cite{VDO}, elliptic waves \cite{BKS}, trapped solitons \cite{ALC} to name but a few. With regard to the (focusing) classical nonlinear Schr\"odinger equation with nonzero boundary conditions with equal absolute values \cite{BK} it is known the so-called modulated instability (Benjamin-Feir instability in the context of water waves) phenomenon, which has been suggested as a possible mechanism for the generation of rogue waves \cite{OOS}.
In \cite{RS2} we present the large-time analysis of problem (\ref{1})-(\ref{2})
in the case of initial data close, in a sense,  to the ``pure step function 
with the step located at $x=0$'': $q_0(x)=0$ for $x<0$ and 
$q_0(x)=A$ for $x>0$. In that case it is shown that 
 the solution  has two qualitatively different asymptotic regions,
 as $t\to\infty$, in the half-plane $-\infty<x<\infty$, $t>0$.
 Namely, for $x<0$ the large-time behavior of the solution is slowly decaying and is described by the Zakharov-Manakov type 
formula \cite{M}, where the power decay rate depends on  $\xi=\frac{x}{4t}$.
On the other hand, for  $x>0$ the solution converges to  constants $c=c(\xi)$, which can be described explicitly in terms of the spectral functions associated to the initial data. Notice that this asymptotic picture is in sharp contrast with  that in the case of the conventional 
(local) nonlinear Schr\"odinger equation 
\begin{equation}\label{NLS}
iq_{t}(x,t)+q_{xx}(x,t)+2q^{2}(x,t)\bar{q}(x,t)=0,
\end{equation}
where there are always three sectors with   different  asymptotic behavior: the Zakharov-Manakov (decaying) sector, the  plane wave sector, and the sector of modulated 
elliptic oscillations \cite{BKS}.

Since the NNLS equation is not translation invariant, the behavior of
the solutions of problem (\ref{1})-(\ref{2}) may depend significantly on 
the details of the shape of the initial data.
The present paper aims to rigorously demonstrate this effect taking the initial data
close to a ``shifted step'', i.e., the pure step function 
with the step located at $x=R$ with some $R>0$: $q_0(x)=0$ for $x<R$ and 
$q_0(x)=A$ for $x>R$.

Indeed, the assumption on the  winding of the argument of a certain 
(spectral) function associated with the initial data adopted in \cite{RS2}
is clearly violated for problem (\ref{1})-(\ref{2}), if the initial data 
has the form of the ``shifted pure step function'', with any $R>0$ provided $A$
is large enough (more precisely, for 
$R>\frac{\pi}{2A}$, see Proposition \ref{a_1ss} and (\ref{windingRH}) below). 
However, due to the presence of the discrete spectrum of the associated linear operator
(from the Lax pair associated with the NNLS equation), we are able 
to modify the transformations of the basic Riemann-Hilbert problem in such a way that the appropriate estimates can be established  in the present case (with
$R>\frac{\pi}{2A}$) as well. These transformations follow the ideas of 
the  nonlinear steepest decent method \cite{DZ} for studying large-time 
behavior of solutions of integrable nonlinear PDEs.

The article is organized as follows. In Section 2 we briefly present the formalism of the inverse scattering transform method, in the form of the associated Riemann--Hilbert 
factorization problem, developed in details in \cite{RS2}, and  discuss the properties of the spectral functions associated to the initial data (\ref{shifted-step}). The 
asymptotic analysis of the Riemann--Hilbert problem is then presented in Section 3,
where the main result of the paper is formulated.

\section{Inverse scattering transform and the Riemann-Hilbert problem}\label{ist}

\subsection{Direct scattering}

The focusing NNLS equation (\ref{1-a}) is a compatibility condition of  two linear differential equations (the  Lax pair) \cite{AMP}
\begin{equation}
\label{LP}
\left\{
\begin{array}{lcl}
\Phi_{x}+ik\sigma_{3}\Phi=U(x,t)\Phi\\
\Phi_{t}+2ik^{2}\sigma_{3}\Phi=V(x,t,k)\Phi\\
\end{array}
\right.
\end{equation}
where $\sigma_3=\left(\begin{smallmatrix} 1& 0\\ 0 & -1\end{smallmatrix}\right)$, $\Phi(x,t,k)$ is a $2\times2$ matrix-valued function, $k\in\mathbb{C}$ is an auxiliary (spectral) parameter, and the matrix coefficients $U(x,t)$ and $V(x,t,k)$ are given in terms of  
$q(x,t)$:	
\begin{equation}
U(x,t)=\begin{pmatrix}
0& q(x,t)\\
-\bar{q}(-x,t)& 0\\
\end{pmatrix},\qquad 
V=\begin{pmatrix}
V_{11}& V_{12}\\
V_{21}& V_{22}\\
\end{pmatrix},
\end{equation}
 where $V_{11}=-V_{22}=iq(x,t)\bar{q}(-x,t)$, $V_{12}=2kq(x,t)+iq_{x}(x,t)$, and
$V_{21}=-2k\bar{q}(-x,t)+i(\bar{q}(-x,t))_{x}$. 

Assuming that there exists $q(x,t)$ satisfying (\ref{1}) and (\ref{2}),
we define  the $2\times2$-valued functions $\Psi_j(x,t,k)$, $j=1,2$, $-\infty<x<\infty$, $0\leq t<\infty$ as the solutions of the linear Volterra integral equations \cite{RS2}:
\begin{subequations}\label{6}
\begin{align}
\label{6-1}
&\Psi_1(x,t,k)=N_-(k)+\int^x_{-\infty}G_-(x,y,t,k)\left(U(y,t)-U_-\right)\Psi_1(y,t,k)e^{ik(x-y)\sigma_3}\,dy,\,
k\in(\mathbb{C}^+,\mathbb{C}^-),\\
\label{6-2}
&\Psi_2(x,t,k)=N_+(k)-\int_x^{\infty}G_+(x,y,t,k)\left(U(y,t)-U_+\right)\Psi_2(y,t,k)e^{ik(x-y)\sigma_3}\,dy,\,
k\in(\mathbb{C}^-,\mathbb{C}^+),
\end{align}
\end{subequations}
where
$$
N_+(k)=
\begin{pmatrix}
1 & \frac{A}{2ik}\\
0 & 1
\end{pmatrix},\,
N_-(k)=
\begin{pmatrix}
1 & 0\\
\frac{A}{2ik} & 1
\end{pmatrix},\,
G_{\pm}(x,y,t,k)=\Phi_{\pm}(x,t,k)[\Phi_{\pm}(y,t,k)]^{-1},
$$ 
with
$
\Phi_{\pm}(x,t,k)=N_{\pm}(k)e^{-(ikx+2ik^2t)\sigma_3},
$
$
U_+=
\begin{pmatrix}
0 & A\\
0 & 0
\end{pmatrix},\,
U_-=
\begin{pmatrix}
0 & 0\\
-A & 0
\end{pmatrix}
$.
Here
$k\in(\mathbb{C}^+,\mathbb{C}^-)$, where $\mathbb{C}^{\pm}=\left\{k\in\mathbb{C}\,|\pm\Im k>0\right\}$, means that the first and the second column of a matrix can be analytically continued into respectively the upper and lower half-plane  as bounded functions.
Then $\Psi_j(x,t,k)e^{-(ikx+2ik^2t)\sigma_3}$, $j=1,2$ are the (Jost) solutions of 
the Lax pair (\ref{LP}) for all 
$k\in\mathbb{R}\setminus\{0\}$ and thus $\Psi_1$ and  $\Psi_2$ are related by
\begin{equation}
\label{9}
\Psi_1(x,t,k)=\Psi_2(x,t,k)e^{-(ikx+2ik^2t)\sigma_3}S(k)e^{(ikx+2ik^2t)\sigma_3},\,\,k\in\mathbb{R}\setminus\{0\},
\end{equation}
where $S(k)$ is the so-called scattering matrix; it can be obtained in terms of the initial data only, evaluating (\ref{6}) at $t=0$: $S(k)=\Psi_2^{-1}(0,0,k)\Psi_1(0,0,k)$.

Since the matrix $U(x,t)$ satisfies the symmetries 
$
\Lambda
\overline{U(-x,t)}
\Lambda^{-1}=U(x,t)
$
with $\Lambda=\left(\begin{smallmatrix} 0& 1\\ 1 & 0\end{smallmatrix}\right)$, it follows that 
 $\Lambda
\overline{\Psi_1(-x,t,-k)}
\Lambda^{-1}=\Psi_2(x,t,k)$ for $k\in\mathbb{R}\setminus\{0\}$. In turn, this implies that the scattering matrix $S(k)$ can be written as
\begin{equation}
S(k)=
\begin{pmatrix}
a_1(k)& -\overline{b(-k)}\\
b(k)& a_2(k)
\end{pmatrix},\qquad k\in\mathbb{R}\setminus\{0\},
\end{equation}
with some $b(k)$, $a_1(k)$, and $a_2(k)$ such that 
$\overline{a_{j}(-k)}=a_j(k)$, $j=1,2$. Moreover, $a_1(k)$, and $a_2(k)$
have analytic continuations into the upper and lower half-planes respectively.
We summarize the properties of the spectral functions in the following proposition
($\overline{\mathbb{C}^{\pm}}=\left\{k\in\mathbb{C}\,|\pm\Im k\geq0\right\}$) \cite{RS2}:
\begin{proposition}\label{properties}
The spectral functions $a_j(k)$, j=1,2, and $b(k)$ have the following properties
\begin{enumerate}
\item 
$a_{1}(k)$ is analytic in  $k\in\mathbb{C}^{+}$
and continuous in 
$\overline{\mathbb{C}^{+}}\setminus\{0\}$;
$a_{2}(k)$ is analytic in $k\in\mathbb{C}^{-}$
and continuous in 
$\overline{\mathbb{C}^{-}}$. 
\item
$a_{j}(k)=1+{O}\left(\frac{1}{k}\right)$, $j=1,2$ and 
$b(k)={O}\left(\frac{1}{k}\right)$ as $k\rightarrow\infty$ (the latter holds
for $k\in{\mathbb R}$).
\item
$\overline{a_{1}(-\bar{k})}=a_1(k)$,  
$k\in\overline{\mathbb{C}^{+}}\setminus\{0\}$; \qquad
$\overline{a_{2}(-\bar{k})}=a_2(k)$,  
$k\in\overline{\mathbb{C}^{-}}$.
\item
$a_{1}(k)a_{2}(k)+b(k)\overline{b(-
\bar{k})}=1$, $k\in{\mathbb R}\setminus\{0\}$ (follows from $\det S(k)=1$).
\item As $k\to 0$, $a_1(k)=\frac{A^2a_2(0)}{4k^2}+O\left(\frac{1}{k}\right)$ and 
$b(k)=\frac{Aa_2(0)}{2ik}+O\left(1\right)$.
\end{enumerate}
\end{proposition}
\begin{remark}\label{remsing}
Item 5 of Proposition \ref{properties} follows from the behavior of $\Psi_j(x,t,k)$ as $k\to0$, which has an additional symmetry \cite{RS2}:
\begin{subequations}
\label{k0}
\begin{align}
\label{k0-a}
&\Psi_1^{(1)}(x,t,k)=\frac{1}{k}
\begin{pmatrix}v_1(x,t)\\v_2(x,t)\end{pmatrix}+O(1),
&\Psi_1^{(2)}(x,t,k)=\frac{2i}{A}\begin{pmatrix}v_1(x,t)\\ v_2(x,t)\end{pmatrix}+O(k),\\
\label{k0-b}
&\Psi_2^{(1)}(x,t,k)=-\frac{2i}{A}
\begin{pmatrix}
\overline{v_2}(-x,t)\\ \overline{v_1}(-x,t)
\end{pmatrix}+O(k),
&\Psi_2^{(2)}(x,t,k)=
-\frac{1}{k}
\begin{pmatrix}\overline{v_2}(-x,t)\\\overline{v_1}(-x,t)
\end{pmatrix}+O(1),
\end{align}
\end{subequations}
with some  $v_j(x,t)$, $j=1,2$, where $\Psi_j^{(i)}(x,t,k)$ denotes the i-th column of 
$\Psi_j(x,t,k)$.
\end{remark}

\subsection{Spectral functions for the ``shifted step'' initial data}
In the case of pure ``shifted step'' initial data
\begin{equation}
\label{shifted-step}
q_0(x)=
\begin{cases}
0,\quad x<R,\\
A,\quad x>R,
\end{cases}
\end{equation}
the associated spectral functions  
can be calculated explicitly:
\begin{subequations}
\label{spf}
\begin{align}
&a_1(k)=1+\frac{A^2}{4k^2}e^{4ikR},\\
&a_2(k)=1,\\
&b(k)=\frac{A}{2ik}e^{2ikR}.
\end{align}
\end{subequations}
Indeed, evaluating (\ref{9}) for $x=-R$ and $t=0$ it follows that the scattering matrix 
$S(k)$ can be determined by
\begin{equation}\label{scatt-st}
S(k)=e^{-ikR\sigma_3}\Psi_2^{-1}(-R,0,k)\Psi_1(-R,0,k)e^{ikR\sigma_3}.
\end{equation}
Taking into account (\ref{shifted-step}),  from (\ref{6}) for $t=0$ we have
\begin{subequations}
\label{psiR}
\begin{align}
&\Psi_1(-R,0,k)=N_-(k),\\
&\Psi_2(-R,0,k)=N_+(k)-\int^R_{-R}G_+(-R,y,0,k)
\begin{pmatrix}
0& -A\\
0& 0
\end{pmatrix}
\Psi_2(y,0,k) e^{-ik(R+y)\sigma_3}\,dy,
\end{align}
\end{subequations}
where $\Psi_2(x,0,k)$ for $x\in\left[-R,R\right]$ solves the  integral equation
\begin{equation}\label{int-psi2}
\Psi_2(x,0,k)=N_+(k)-\int^R_x G_+(x,y,0,k)
\begin{pmatrix}
0& -A\\
0& 0
\end{pmatrix}
\Psi_2(y,0,k)
e^{ik(x-y)\sigma_3}\,dy,\quad x\in[-R, R].
\end{equation}
Direct calculations show that
\begin{equation*}
G_+(x,y,0,k)=
\begin{pmatrix}
e^{-ik(x-y)}& \frac{A}{2ik}\left(e^{ik(x-y)}-e^{-ik(x-y)}\right)\\
0& e^{ik(x-y)}
\end{pmatrix},
\end{equation*}
and, therefore, the solution  $\Psi_2(x,0,k)$ of (\ref{int-psi2})
is given by  
\begin{equation}\label{psi2R}
\Psi_2(x,0,k)=
\begin{pmatrix}
1 & \frac{A}{2ik}e^{2ik(R-x)}\\
0 & 1
\end{pmatrix},\quad x\in\left[-R,R\right].
\end{equation}
Substituting (\ref{psiR}) and (\ref{psi2R}) into (\ref{scatt-st}) one obtains (\ref{spf}).

The locations  of zeros of $a_1(k)$ in $\overline{{\mathbb{C}^{+}}}$, which
 clearly depend on $A$ and $R$,
and the behavior of the argument of $a_1(k)$ for  $k\in \mathbb R$
are described in the following

\begin{proposition}\label{a_1ss}

\begin{enumerate}[(i)]
\item For $0<R<\frac{\pi}{2A}$,   $a_1(k)$ 
has one simple zero in $\overline{\mathbb{C}^{+}}$ at  $k=ik_0$, $k_0>0$, where $k_0$ is the unique solution of the  transcendental equation
\begin{equation}
\label{transc}
k=\frac{A}{2}e^{-2kR},\quad k\in\mathbb{R}.
\end{equation}
Moreover, for all $\xi>0$,
\begin{equation}\label{winding0}
\int_{-\infty}^{-\xi}d\arg a_1(k)\in(-\pi,\pi).
\end{equation}
\item For $\frac{(2n-1)\pi}{2A}<R<\frac{(2n+1)\pi}{2A}$, $n\in\mathbb{N}$, 
 $a_1(k)$ has the following properties:
\begin{itemize}
	\item 
	$a_1(k)$ has $2n+1$ simple zeros in $\overline{\mathbb{C}^{+}}$:
$\{ik_0;\{p_j, -\overline{p}_j\}_{j=1}^{n}\}$. Here $k_0>0$ 
 is the solution of (\ref{transc}),
 $\{\Re p_j\}_{j=1}^{n}$ are the ordered set of solutions of 
 equation 
\begin{equation}
\label{transc2}
k=\pm\frac{A}{2}\sin (2kR) e^{2kR\cot (2kR)},
\end{equation}
considered for $k<0$
(see also Figure \ref{zer}),
and 
\begin{equation}\label{Imp_j}
\Im p_j=-\Re p_j\cot (2\Re p_jR),\quad j=\overline{1,n}.
\end{equation}
Notice that 
\begin{equation}
\label{interv}
\Re p_j\in\left(-\frac{j\pi}{2R},-\frac{(2j-1)\pi}{4R}\right),
\quad j=\overline{1,n}.
\end{equation}
\item
Let $\omega_0=0$, $\omega_j=\frac{(2j-1)\pi}{4R}$ for $j=\overline{1,n}$, and 
$\omega_{n+1}=\infty$. Then  
\begin{subequations}\label{winding}
\begin{align}
&\int_{-\infty}^{-\omega_{n-j+1}}d\arg a_1(k) = (2j-1)\pi,\quad j=\overline{1,n},\\
&\int_{-\infty}^{-\xi}d\arg a_1(k)\in((2j-1)\pi,(2j+1)\pi),\quad -\omega_{n-j+1}<-\xi<-\omega_{n-j},\quad j=\overline{0,n}.
\end{align}
\end{subequations}

\end{itemize}

\item If $R=\frac{(2n+1)\pi}{2A}$ for some $n\in\mathbb{N}\cup\{0\}$, then  
$a_1(k)$ has $2n+3$ simple zeros in $\overline{\mathbb{C}^{+}}$ at 
$\{\pm\frac{A}{2}$, $ik_0$,  $\{p_j, -\overline{p}_j\}_{j=1}^{n}\}$, where $k_0>0$ is the  solution of  (\ref{transc}),   $\Re p_j$ ($j=\overline{1,n}$) are the solutions of (\ref{transc2}), and $\Im p_j$ are determined by 
 (\ref{Imp_j}).
\end{enumerate}
\begin{figure}[h]
\label{zer}
\begin{minipage}[h]{0.49\linewidth}
\centering{\includegraphics[width=0.99\linewidth]{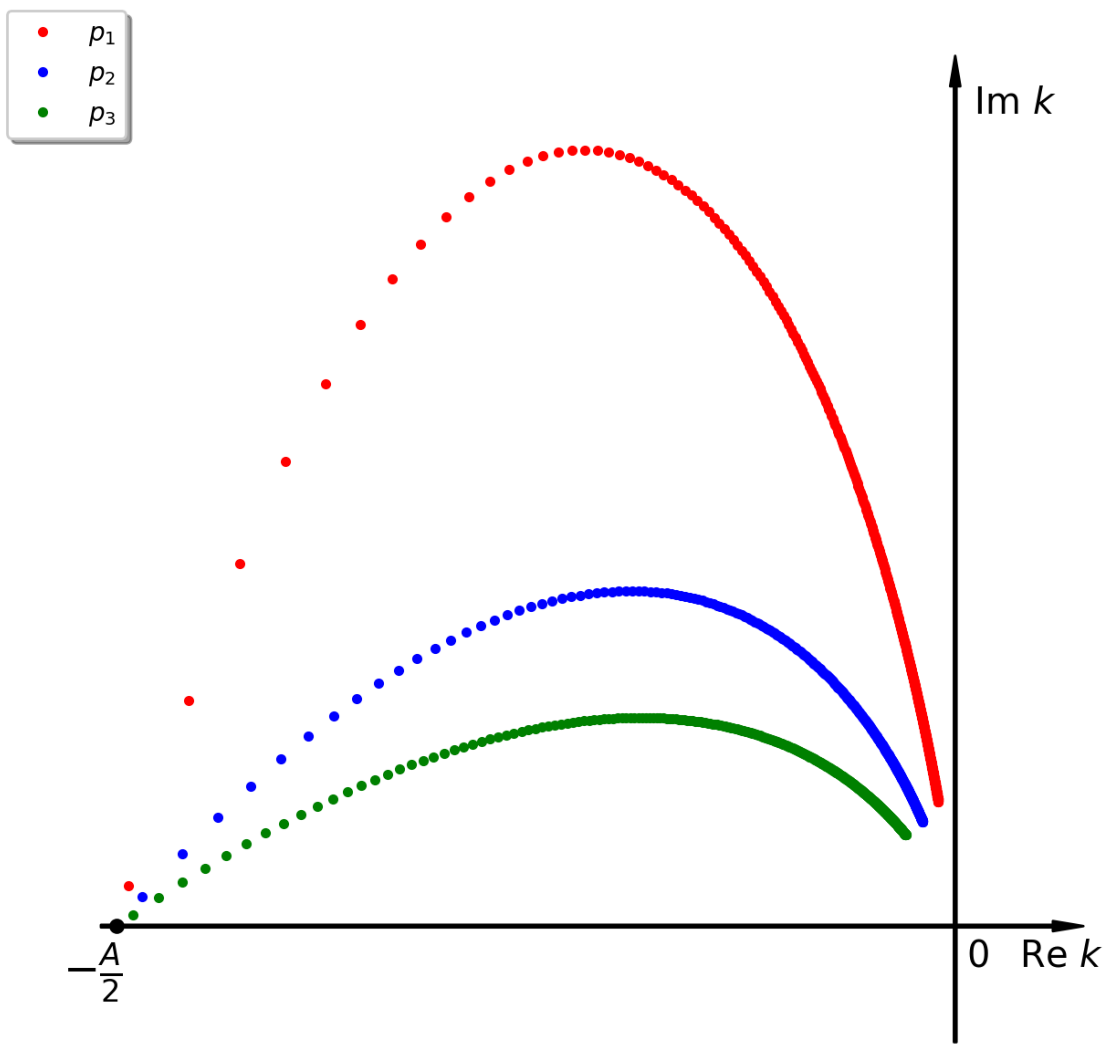}}
\end{minipage}
\hfill
\begin{minipage}[h]{0.49\linewidth}
\centering{\includegraphics[width=0.79\linewidth]{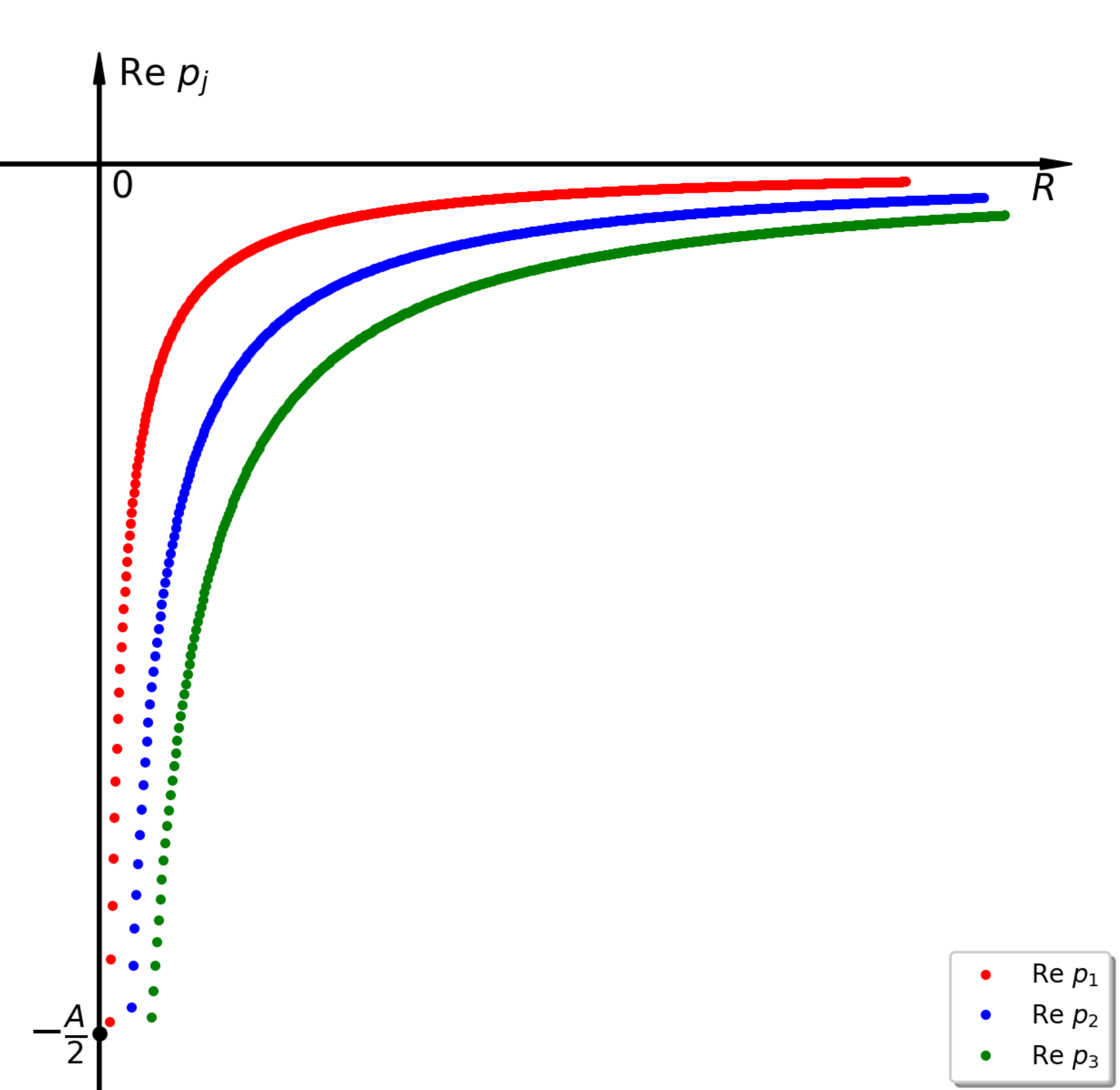}}
\vfill
\centering{\includegraphics[width=0.79\linewidth]{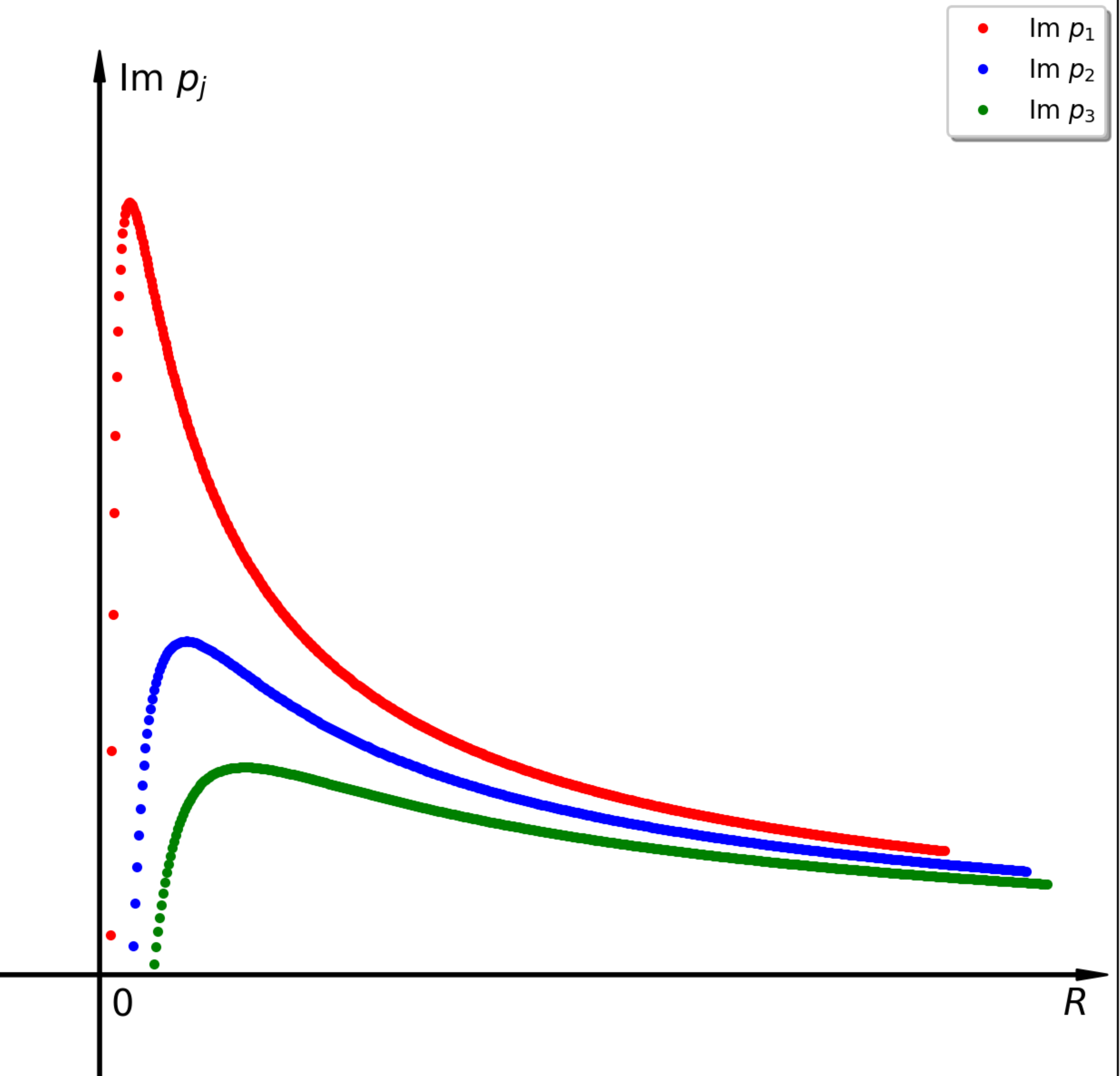}}
\end{minipage}
\caption{``Evolution'' of the zeros $p_j$, $j=1,2,3$, as $R\to\infty$.}
\end{figure}
\end{proposition}

\begin{proof}

 Observe that the equation $a_1(k)=0$ is equivalent to the system
\begin{equation}
\label{syst}
\begin{cases}
k_1=\mp\frac{A}{2}\sin(2k_1R)e^{-2k_2R}\\
k_2=\pm\frac{A}{2}\cos(2k_1R)e^{-2k_2R}
\end{cases},
\end{equation}
where $k=k_1+ik_2$, $k\in\overline{\mathbb{C}^{+}}\setminus\{0\}$. Due to the symmetry relation $a_1(k)=\overline{a_1(-\bar{k})}$ it is sufficient to consider (\ref{syst}) 
for $k_1\geq0$ only.

(i) Assuming $k_1=0$, the system (\ref{syst}) reduces to the equations 
$k_2=\pm\frac{A}{2}e^{-2k_2R}$ and thus $a_1(k)$ has exactly one purely imaginary 
simple zero (with $k_2>0$) for all $R>0$ and $A>0$,  and its imaginary part is the solution of  (\ref{transc}).

(ii)  Assuming $k_2=0$, the second equation in (\ref{syst}) implies that 
 $k_1$ must be equal to $\frac{\pi+2\pi n}{4R}$ with some $n\in\mathbb{N}\cup\{0\}$.
But then,  from the first equation in (\ref{syst}) we conclude that $k_1=\frac{A}{2}$. 
Therefore, $k=\frac{A}{2}$ is a simple zero of $a_1(k)$ if and only if there exists $n\in\mathbb{N}\cup\{0\}$ such that $\pi+2\pi n=2AR$.

(iii) Now, let's look at the location of  zeros of $a_1(k)$ in the open quarter plane $k_1>0$, $k_2>0$. Dividing the  equations in (\ref{syst}) sidewise we arrive at  (cf. (\ref{Imp_j}))
\begin{equation}\label{cot}
k_2=-k_1\cot(2k_1R),\quad k_1\neq\frac{n\pi}{4R},\,
n\in\mathbb{N},
\end{equation}
from which  we conclude (cf. (\ref{interv})) that 
\begin{equation}\label{domain}
k_1\in\left(\frac{(2n-1)\pi}{4R},\frac{n\pi}{2R}\right),\quad n\in\mathbb{N}.
\end{equation}
Substituting (\ref{cot}) into the first equation in (\ref{syst}) and taking into account the sign of $\sin(2k_1R)$ for $k_1$ satisfying (\ref{domain}), we obtain an equation 
for $k_1$ in the form 
\begin{subequations}\label{k_1}
\begin{equation}\label{k1-1}
k_1=\frac{A}{2}\sin(2k_1R)e^{2k_1R\cot(2k_1R)}\quad \text{for}\ 
k_1\in\left(\frac{(4n-3)\pi}{4R},\frac{(2n-1)\pi}{2R}\right),\quad n\in\mathbb{N},
\end{equation}
or 
\begin{equation}\label{k1-2}
k_1=-\frac{A}{2}\sin(2k_1R)e^{2k_1R\cot(2k_1R)} \quad \text{for}\ 
 k_1\in\left(\frac{(4n-1)\pi}{4R},\frac{n\pi}{R}\right),\quad n\in\mathbb{N}.
\end{equation}
\end{subequations}
Since the r.h.s. of (\ref{k1-1}) and (\ref{k1-2}) monotonically decrease in
the  corresponding intervals for $k_1$, it follows that equations (\ref{k_1}) have no solutions for $0<R\leq\frac{\pi}{2A}$, whereas  for 
$\frac{(2n-1)\pi}{2A}<R\leq\frac{(2n+1)\pi}{2A}$ equations (\ref{k_1}) 
have  $n$ simple solutions $\{k_{1,j}\}_{j=1}^{n}$ such that  $k_{1,j}\in\left(\frac{(2j-1)\pi}{4R},\frac{j\pi}{2R}\right)$, 
$j=\overline{1,n}$ (cf.  (\ref{interv})).

Concerning the   winding properties of $\arg a_1(k)$, (\ref{winding0}) 
(for $0<R<\frac{\pi}{2A}$) follows from the inequality
\begin{equation}
\frac{A^2}{4k_{(m)}^{2}}e^{4ik_{(m)}R}>-1,\quad \text{where}\ 
k_{(m)}=\frac{(1-2m)\pi}{4R},\quad m\in\mathbb{N},
\end{equation}
whereas (\ref{winding}) (for $\frac{(2n-1)\pi}{2A}<R<\frac{(2n+1)\pi}{2A}$) follows from
\begin{subequations}
\begin{align}
&\frac{A^2}{4k_{(m)}^2}e^{4ik_{(m)}R}>-1\quad \text{for}\ 
k_{(m)}=\frac{(1-2m)\pi}{4R},\quad
m\in\mathbb{N},\quad
m>n,\\
&\frac{A^2}{4k_{(m)}^2}e^{4ik_{(m)}R}<-1 \quad \text{for}\ 
k_{(m)}=\frac{(1-2m)\pi}{4R},\quad
m\in\mathbb{N},\quad
m\leq n.
\end{align}
\end{subequations}
\end{proof}
\subsection{The basic Riemann-Hilbert problem and inverse scattering}
\label{bRH}

The Riemann--Hilbert formalism of the Inverse Scattering Transform method
 is based on constructing a piece-wise meromorphic,
$2\times 2$-valued function in the $k$-complex plane, which has the prescribed jump across a some contour in the complex plane and prescribed conditions at singular points (in case they are present).

The analytic properties of the Jost solutions $\Psi_j$ suggest 
defining the $2\times 2$-valued function $M(x,t,k)$, piece-wise meromorphic relative to 
$\mathbb R$, as follows \cite{RS2}:

\begin{equation}
\label{DM}
M(x,t,k)=
\left\{
\begin{array}{lcl}
\left(\frac{\Psi_1^{(1)}(x,t,k)}{a_{1}(k)},\Psi_2^{(2)}(x,t,k)\right),\quad k\in\mathbb{C}^+\setminus\{0\},\\
\left(\Psi_2^{(1)}(x,t,k),\frac{\Psi_1^{(2)}(x,t,k)}{a_{2}(k)}\right),\quad k\in\mathbb{C}^-\setminus\{0\}.\\
\end{array}
\right.
\end{equation}
Then the scattering relation (\ref{9}) implies that the boundary values 
$M_\pm(x,t,k) = \underset{k'\to k, k'\in {\mathbb C}^\pm}{\lim} M(x,t,k')$, $k\in\mathbb R$
 satisfy the  multiplicative jump condition
\begin{equation}\label{jr}
M_+(x,t,k)=M_-(x,t,k)J(x,t,k),\qquad k\in\mathbb{R}\setminus\{0\},
\end{equation}
where
\begin{equation}\label{jump}
J(x,t,k)=
\begin{pmatrix}
1+r_{1}(k)r_{2}(k)& r_{2}(k)e^{-2ikx-4ik^2t}\\
r_1(k)e^{2ikx+4ik^2t}& 1
\end{pmatrix},
\end{equation}
with the reflection coefficients defined by
\begin{equation}\label{r12}
r_1(k):=\frac{b(k)}{a_1(k)},\quad r_2(k):=\frac{\overline{b(-k)}}{a_2(k)}.
\end{equation}
Observe that by the determinant relation (see item 4 in Proposition \ref{properties}) we have
\begin{equation}\label{scalrel}
1+r_1(k)r_2(k)=\frac{1}{a_1(k)a_2(k)}.
\end{equation}
Moreover, 
\begin{equation}
M(x,t,k)\to I,\qquad k\to\infty,
\end{equation}
where  $I$ is the $2\times2$ identity matrix.

Taking into account the singularities of 
$\Psi_j(x,t,k)$, $j=1,2$ and  $a_1(k)$  at  $k=0$ (see Proposition \ref{properties} and Remark \ref{remsing}), the behavior of $M(x,t,k)$ at $k=0$ can be described as follows:
\begin{subequations}\label{z}
	\begin{align}
	\label{+i0}
	& M_+(x,t,k)=
	\begin{pmatrix}
	\frac{4}{A^2a_2(0)}v_1(x,t)& -\overline{v_2}(-x,t)\\
	\frac{4}{A^2a_2(0)}v_2(x,t)& -\overline{v_1}(-x,t)
	\end{pmatrix}
	(I+O(k))
	\begin{pmatrix}
	k& 0\\
	0& \frac{1}{k}
	\end{pmatrix}
	,\quad k\rightarrow +i0,\\
	\label{-i0}
	& M_-(x,t,k)=\frac{2i}{A}
	\begin{pmatrix}
	-\overline{v_2}(-x,t)& \frac{v_1(x,t)}{a_2(0)}\\
	-\overline{v_1}(-x,t)& \frac{v_2(x,t)}{a_2(0)}
	\end{pmatrix}
	+O(k)
	,\quad k\rightarrow -i0,
	\end{align}
\end{subequations}

Now, 
being motivated by the properties of $a_1(k)$ and $a_2(k)$ in the case of ``shifted step'' initial data
(see Proposition \ref{a_1ss}), we make 
 the following additional assumptions on $a_1(k)$ and $a_2(k)$ in the case of
general step-like initial data satisfying (\ref{shifted-step-gen}):
\begin{description}
\item [Assumptions A:] 
\item{\textbf{({a-1})}} 
$a_1(k)$ has $2n+1$, $n\in\mathbb{N}\cup\{0\}$, simple zeros 
in $\overline{{\mathbb C}^+}\setminus\{0\}$
at $k=ik_0$ with $k_0>0$, at $\{p_j\}_{j=1}^{n}$ and at $\{-\overline{p}_j\}_{j=1}^{n}$ with $\Im p_j>0$ and 
$\Re p_n<\dots<\Re p_1<0$. 
\item{\textbf{({a-2})}}
$a_2(k)$ has no zeros in $\overline{\mathbb{C}^-}$.
\item{\textbf{(b)}} 
There are $\omega_m>0$, $m=\overline{0,n+1}$, such that
\begin{equation}\label{order}
-\infty=-\omega_{n+1}<\Re p_n<-\omega_n<\Re p_{n-1}<-\omega_{n-1}<\dots<\Re p_1<-\omega_1<\omega_0=0,
\end{equation}
\begin{subequations}\label{windingRH}
\begin{equation}
\int_{-\infty}^{-\omega_{n-m+1}}d\arg\left(a_1(k)a_2(k)\right) = (2m-1)\pi,
	\quad m=\overline{1,n},
\end{equation}
and 
\begin{equation}
\int_{-\infty}^{-\xi}d\arg\left(a_1(k)a_2(k)\right)
\in((2m-1)\pi,(2m+1)\pi),\quad -\omega_{n-m+1}<-\xi<-\omega_{n-m},\quad m=\overline{0,n}.
\end{equation}
\end{subequations}
\end{description}

In accordance with this assumption, $M(x,t,k)$ satisfies the residue conditions:
\begin{subequations}\label{resin}
\begin{align}
\underset{k=ik_0}{\operatorname{Res}} M^{(1)}(x,t,k)&=
\frac{\gamma_0}{\dot{a}_1(ik_0)}e^{-2k_0x-4ik_0^2t}M^{(2)}(x,t,ik_0),\quad
|\gamma_0|=1,\\
\underset{k=p_j}{\operatorname{Res}} M^{(1)}(x,t,k)&=
\frac{\eta_j}{\dot a_1(p_j)}
e^{2ip_jx+4ip_j^2t}M^{(2)}(x,t,p_j),\quad j=\overline{1,n},\\
\underset{k=-\overline{p}_j}{\operatorname{Res}} M^{(1)}(x,t,k)&=
\frac{1}{\bar\eta_j\dot a_1(-\overline{p}_j)}
e^{-2i\overline{p}_jx+4i\overline{p}_j^2t}M^{(2)}(x,t,-\overline{p}_j),
\quad j=\overline{1,n},
\end{align}
\end{subequations}
where $\gamma_0$ and $\eta_j$ come from the relations $\Psi_1^{(1)}(0,0,ik_0)=\gamma_0\Psi_2^{(2)}(0,0,ik_0)$ and $\Psi_1^{(1)}(0,0,p_j)=\eta_j\Psi_2^{(2)}(0,0,p_j)$
for the eigenfunctions of the first equation from the Lax pair (\ref{LP}).


\begin{remark}
If $b(k)$ allows analytical continuation  into a
sufficiently large band in the complex plane, the norming constants take the form: 
$$
\gamma_0=b(ik_0),\,\eta_j=b(p_j).
$$
\end{remark}


\begin{description}
\item [Basic Riemann--Hilbert Problem:] 

Given (i) $b(k)$ for $k\in{\mathbb R}$, (ii) $a_j(k)$, $j=1,2$ having the properties of 
Proposition \ref{properties} and satisfying Assumptions A,
with $\{ik_0, \{p_j, -\bar p_j\}_1^n\}$ being the zeros of $a_1(k)$ in ${\mathbb C}^+$, and 
(iii) $\gamma_0$ and $\{\eta_j\}_1^n$, 
find the $2\times 2$-valued function $M(x,t,k)$, piece-wise meromorphic in $k$ relative to  
$\mathbb{R}$ and satisfying the following conditions:

\begin{description}
\item [(i)] Jump conditions. The boundary values $M_{\pm}(x,t,k)=M(x,t,k\pm i0)$,  
$k\in\mathbb{R}\setminus\{0\}$ satisfy the condition
\begin{equation}\label{jRH}
M_+(x,t,k)=M_-(x,t,k)J(x,t,k),\qquad k\in\mathbb{R}\setminus\{0\},
\end{equation}
where the jump matrix $J(x,t,k)$ is given by (\ref{jump}), with $r_j(k)$  given
in terms of $b(k)$ and $a_j(k)$
by (\ref{r12}).
\item[(ii)] Normalization at $k=\infty$:
$$
M(x,t,k)=I+O(k^{-1}),\qquad k\to\infty.
$$
\item [(iii)] Residue conditions (\ref{resin}).
\item [(iv)] \textit{Pseudo-residue} conditions at $k=0$: $M(x,t,k)$ satisfies 
(\ref{z}), 
where $v_j(x,t)$, $j=1,2$ are some (not prescribed) functions.
\end{description}

Assume that the RH problem (i)--(iv) has a solution $M(x,t,k)$. 
Then the solution of the Cauchy problem (\ref{1}), (\ref{2}) is given in terms of the (12) and (21) entries of $M(x,t,k)$ as follows:
\begin{equation}\label{sol}
q(x,t)=2i\lim_{k\to\infty}kM_{12}(x,t,k),
\end{equation}
and
\begin{equation}\label{sol1}
q(-x,t)=-2i\lim_{k\to\infty}k\overline{M_{21}(x,t,k)}.
\end{equation}
\end{description}

\begin{remark}
We coin a term \textit{Pseudo-residue} condition for (\ref{z}), since after appropriate transformations of the basic Riemann-Hilbert problem it turns into a conventional residue condition (see the RH problem (\ref{RHhat})--(\ref{zc}) below).
\end{remark}

The solution of the RH problem is unique, if exists. Indeed, if $M$ and $\tilde M$ are 
two solutions, then
conditions (\ref{z}) provide the boundedness of $M \tilde M^{-1}$ at $k=0$
and, therefore, the applicability of standard arguments based of the Liouville theorem.

\begin{remark} 
From  (\ref{sol}) and (\ref{sol1}) it follows  that in order to present the solution of (\ref{1}), (\ref{2}) for all $x\in \mathbb R$, it is sufficient to have the solution the RH problem for, say, $x\ge 0 $ only.
\end{remark}

\section{The long-time asymptotics}\label{asympt}

In this section we study the long-time asymptotics of the solution of the Cauchy problem 
(\ref{1}),
(\ref{2}). Our analysis is based on the adaptation of 
 the nonlinear steepest-decent method \cite{DZ} to the (oscillatory) RH problem (i)--(iv). 

\subsection{Jump factorizations}

Introduce the  variable $\xi:=\frac{x}{4t}$ and the phase function
\begin{equation}\label{theta}
\theta(k,\xi)=4k\xi+2k^2.
\end{equation}
The jump matrix (\ref{jump}) allows, similarly to  \cite{RS}, two triangular factorizations:
\begin{subequations}\label{tr}
\begin{align}
\label{tr1}
J(x,t,k)&=
\begin{pmatrix}
1& 0\\
\frac{r_1(k)}{1+ r_1(k)r_2(k)}e^{2it\theta}& 1\\
\end{pmatrix}
\begin{pmatrix}
1+ r_1(k)r_2(k)& 0\\
0& \frac{1}{1+ r_1(k)r_2(k)}\\
\end{pmatrix}
\begin{pmatrix}
1& \frac{ r_2(k)}{1+ r_1(k)r_2(k)}e^{-2it\theta}\\
0& 1\\
\end{pmatrix}
\\
\label{tr2}
&=
\begin{pmatrix}
1& r_2(k)e^{-2it\theta}\\
0& 1\\
\end{pmatrix}
\begin{pmatrix}
1& 0\\
r_1(k)e^{2it\theta}& 1\\
\end{pmatrix}.
\end{align}
\end{subequations}
In view of  (\ref{sol}) and (\ref{sol1}), we will study the RH problem for $\xi>0$ only.
Since the  phase function $\theta(k,\xi)$ is the same as in the case of the local NLS,
its  signature table (see Figure \ref{signtable}) suggests us to follow the standard steps \cite{DIZ, DZ}
involving getting rid of the diagonal factor in (\ref{tr1}) and the deformation of the 
original RH (relative to the real axis) to the new one relative to a cross, where the jump matrix converges,
as $t\to\infty$,
to the identity matrix uniformly away from any vicinity of the stationary phase point $k=-\xi$.

\begin{figure}
\begin{minipage}[h]{0.99\linewidth}
\centering{\includegraphics[width=0.4\linewidth]{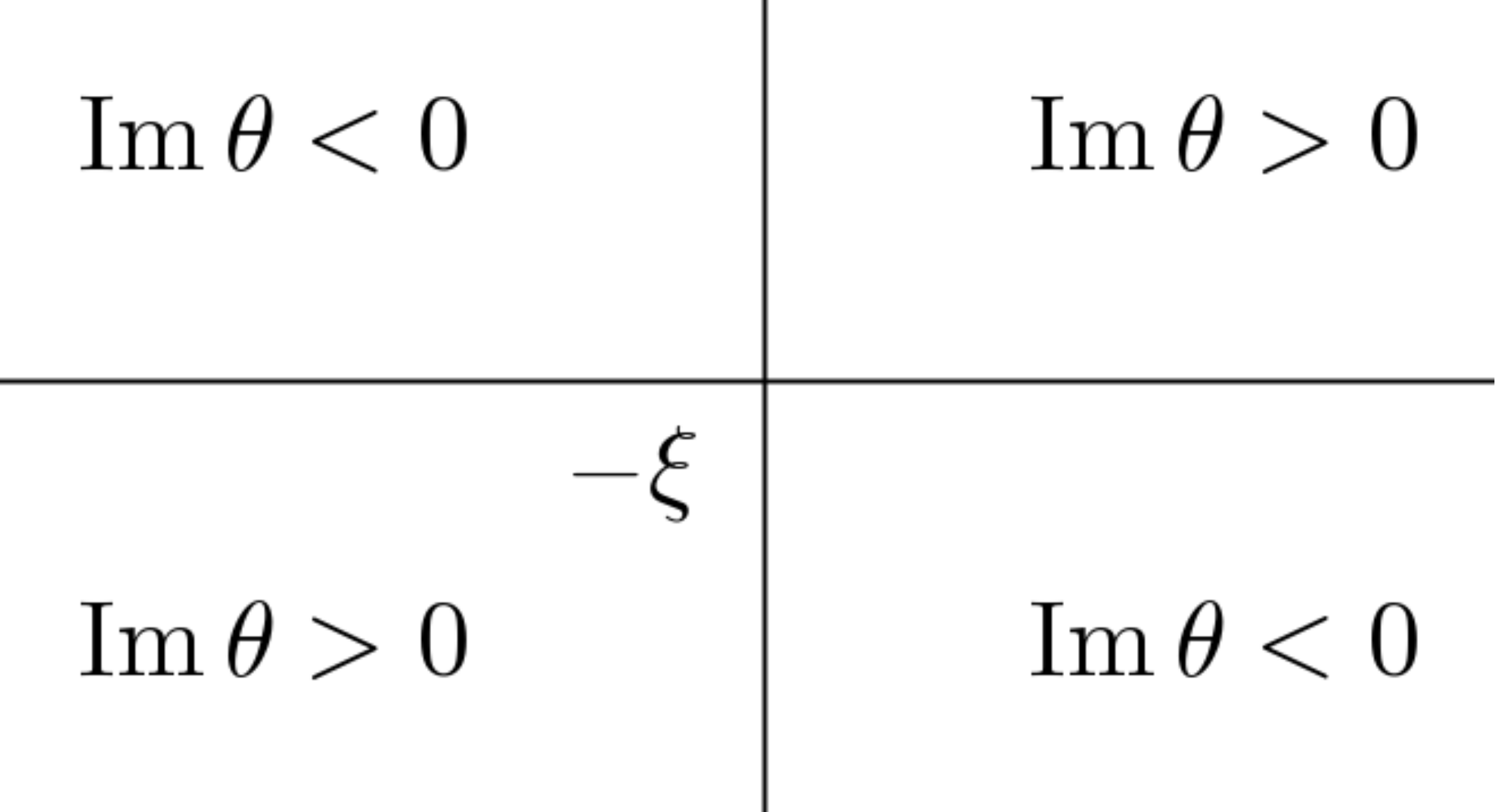}}
\caption{Signature table}

\label{signtable}
\end{minipage}
\end{figure}

In order to get rid of the diagonal factor in a factorization like (\ref{tr1}), 
one usually \cite{DIZ, DZ}  introduces a (scalar) 
function $\delta(k)$ that solves the scalar RH problem with the 
jump condition $\delta_+(k)=\delta_-(k)(1+r_1(k)r_2(k))$ for 
$k\in(-\infty,-\xi)$. In the case $1+r_1(k)r_2(k)>0$ for all $k\in \mathbb R$,
as it takes place for the local NLS equation, $\delta(k)$ is defined via a Cauchy integral
involving $\ln (1+r_1(k)r_2(k))$. However, in the case of the nonlocal NLS equation,
the values of $1+r_1(k)r_2(k)$ are, in general, complex, which would lead 
to a strong singularity of $\delta(k)$ at $k=-\xi$. In order to avoid this,
we proceed as follows:
\begin{enumerate}
	\item 
	First, define   some ``partial functions delta'':

\begin{subequations}
\begin{align}
&\delta_s(k)=\delta_s(k;\omega_{n-s},\omega_{n-s+1})=\exp\left\{
\frac{1}{2\pi i}\int_{-\omega_{n-s+1}}^{-\omega_{n-s}}
\frac{\ln(1+r_1(\zeta)r_2(\zeta))}{\zeta-k}\,d\zeta
\right\},\quad s=\overline{0,m-1},\\
&\delta_m(k)=\delta_m(k,\xi;\omega_{n-m+1})=\exp\left\{
\frac{1}{2\pi i}\int_{-\omega_{n-m+1}}^{-\xi}
\frac{\ln(1+r_1(\zeta)r_2(\zeta))}{\zeta-k}\,d\zeta
\right\},
\end{align}
\end{subequations}
where $-\xi\in(-\omega_{n-m+1}, -\omega_{n-m})$, $m=\overline{0,n}$ and the following branches of logarithm are chosen
for $s=\overline{0,m}$ (notice that since we deal with $\xi>0$, the behavior of $r_j(k)$ at $k=0$ does not affect $\delta_m(k)$):
\begin{equation}\label{branch}
\ln(1+r_1(\zeta)r_2(\zeta))=\ln|1+r_1(\zeta)r_2(\zeta)|+i\left(
\int_{-\infty}^{\zeta}\,d\arg(1+r_1(z)r_2(z))+2\pi s
\right).
\end{equation}

\item 
Second, 
define
	\begin{equation}
\delta(k,\xi;\{\omega_{n-s}\}_{s=0}^{m-1}):=
\prod\limits_{s=0}^{m}\delta_s(k).
\end{equation}
\end{enumerate}

In this way, we have that 
$\delta(k)= \delta(k,\xi;\{\omega_{n-s}\}_{s=0}^{m-1})$
solves, for each $-\xi\in (-\omega_{n-m+1}, -\omega_{n-m})$,
$m=\overline{0,n}$ (see (\ref{order})),
 the scalar RH problem
\begin{subequations}\label{13}
\begin{align}
&\delta_+(k)=\delta_-(k)(1+r_1(k)r_2(k)),&&
k\in(-\infty,-\xi)\setminus\{-\omega_{n-s}\}_{s=0}^{m-1},\\
&\delta(k)\rightarrow 1,&&
k\rightarrow\infty.
\end{align}
\end{subequations}
Moreover, it has  particular singularities  at $k=-\omega_{n-s}$, $s=\overline{0,m-1}$, and
 $k=-\xi$. 
Namely, adopting the convention that $\prod\limits_{s=m_1}^{m_2}F_s=1$ if $m_1>m_2$, we  have 
\begin{equation}\label{delta}
\delta(k,\xi;\{\omega_{n-s}\}_{s=0}^{m-1})=
(k+\xi)^{i\nu(-\xi)}\prod\limits_{s=0}^{m-1}(k+\omega_{n-s})^{-1}
\exp\left\{\sum\limits_{s=0}^{m}\chi_s(k)\right\},
\end{equation}
with 
\begin{subequations}
\begin{align}
&\chi_s(k)=-\frac{1}{2\pi i}\int_{-\omega_{n-s+1}}^{-\omega_{n-s}}
\ln(k-\zeta)\,d_{\zeta}\ln(1+r_1(\zeta)r_2(\zeta)),\quad s=\overline{0,m-1},\\
&\chi_m(k)=-\frac{1}{2\pi i}\int_{-\omega_{n-m+1}}^{-\xi}
\ln(k-\zeta)\,d_{\zeta}\ln(1+r_1(\zeta)r_2(\zeta)),
\end{align}
\end{subequations}
and 
\begin{equation}\label{nu}
\nu(-\xi)=-\frac{1}{2\pi}\ln|1+r_1(-\xi)r_2(-\xi)|
-\frac{i}{2\pi}\left(\int_{-\infty}^{-\xi}\,
d\arg(1+r_1(\zeta)r_2(\zeta))+2\pi m\right),
\end{equation}
so that (see  Assumptions A(b) (\ref{windingRH}) and relation (\ref{scalrel}))
$\Im\nu(-\xi)$ satisfies the inequalities 
\[
-\frac{1}{2}< \Im\nu(-\xi) <\frac{1}{2}.
\]

\begin{remark}\label{rem6}
In our asymptotic analysis, it is important to have 
$\Im\nu(-\xi)\in(-\frac{1}{2},\frac{1}{2})$. 
This property will provide the convergence, as $t\to\infty$,
of the solution of the deformed Riemann-Hilbert problem (relative to the cross
centered at $k=-\xi$) to
 the identity matrix, see subsection \ref{RHdeform} below.
\end{remark}

Now we define
\begin{equation}
\tilde{M}(x,t,k)=M(x,t,k)
\delta^{-\sigma_3}(k,\xi;\{\omega_{n-s}\}_{s=0}^{m-1}),
\end{equation} 
and notice that $\tilde{M}$
 satisfies the conditions
\begin{subequations}\label{tildeM}
\begin{align}
\label{14.1}
&\tilde{M}_+(x,t,k)=\tilde{M}_-(x,t,k)\tilde{J}(x,t,k), && k\in\mathbb{R}\setminus\left(\{-\omega_{n-s}\}_{s=0}^{m-1}\cup\{0\}
\right),\\
\label{14.1-norm}
&\tilde{M}(x,t,k)\rightarrow I, &&k\rightarrow\infty,
\end{align}
\end{subequations}
where (for simplicity we drop all arguments of $\delta$ except $k$ and $\xi$)
\begin{equation}
\label{as4}
\tilde{J}(x,t,k)=
\begin{cases}
\begin{pmatrix}
1& 0\\
\frac{r_1(k)\delta_-^{-2}(k,\xi)}{1+r_1(k)r_2(k)}e^{2it\theta}& 1\\
\end{pmatrix}
\begin{pmatrix}
1& \frac{r_2(k)\delta_+^{2}(k,\xi)}{1+r_1(k)r_2(k)}e^{-2it\theta}\\
0& 1\\
\end{pmatrix},\, k\in(-\infty,-\xi)\setminus\{-\omega_{n-s}\}_{s=0}^{m-1},
\\
\begin{pmatrix}
1& r_2(k)\delta^2(k,\xi)e^{-2it\theta}\\
0& 1\\
\end{pmatrix}
\begin{pmatrix}
1& 0\\
r_1(k)\delta^{-2}(k,\xi)e^{2it\theta}& 1\\
\end{pmatrix},\, k\in(-\xi,\infty)\setminus\{0\},
\end{cases}
\end{equation}
as well as  the residue conditions
\begin{subequations}
\begin{align}
&\underset{k=ik_0}{\operatorname{Res}}\tilde M^{(1)}(x,t,k)=
\frac{\gamma_0}{\dot{a}_1(ik_0)\delta^{2}(ik_0,\xi)}e^{-2k_0x-4ik_0^2t}
\tilde M^{(2)}(x,t,ik_0),\quad|\gamma_0|=1,\\
&\underset{k=p_j}{\operatorname{Res}}\tilde M^{(1)}(x,t,k)=
\frac{\eta_j}{\dot a_1(p_j)\delta^{2}(p_j,\xi)}
e^{2ip_jx+4ip_j^2t}\tilde M^{(2)}(x,t,p_j),\quad j=\overline{1,n},\\
&\underset{k=-\overline{p}_j}{\operatorname{Res}}\tilde M^{(1)}(x,t,k)=
\frac{1}{\bar\eta_j\dot a_1(-\overline{p}_j)
\delta^{2}(-\overline{p}_j,\xi)}
e^{-2i\overline{p}_jx+4i\overline{p}_j^2t}\tilde M^{(2)}(x,t,-\overline{p}_j),
\quad j=\overline{1,n},
\end{align}
\end{subequations}
related  to the zeros in $\mathbb{C}^{+}$, 
and the pseudo-residue conditions at $k=0$ 
\begin{subequations}\label{14}
\begin{align}
\label{14.2}
&\tilde{M}_+(x,t,k)=
\begin{pmatrix}
\frac{4v_1(x,t)}{A^2a_2(0)\delta(0,\xi)}& -\delta(0,\xi)\overline{v_2}(-x,t)\\
\frac{4v_2(x,t)}{A^2a_2(0)\delta(0,\xi)}& -\delta(0,\xi)\overline{v_1}(-x,t)
\end{pmatrix}
(I+O(k))
\begin{pmatrix}
k& 0\\
0& \frac{1}{k}
\end{pmatrix},\,
k\rightarrow +i0,\\
\label{14.3}
&\tilde{M}_-(x,t,k)=\frac{2i}{A}
\begin{pmatrix}
\frac{-\overline{v_2}(-x,t)}{\delta(0,\xi)}& \delta(0,\xi)\frac{v_1(x,t)}{a_2(0)}\\
\frac{-\overline{v_1}(-x,t)}{\delta(0,\xi)}& \delta(0,\xi)\frac{v_2(x,t)}{a_2(0)}
\end{pmatrix}
+O(k),\,
k\rightarrow -i0,
\end{align}
\end{subequations}
and at $k=-\omega_{n-s}$, $s=\overline{0,m-1}$:
\begin{equation}\label{p-res-2}
\tilde{M}_{\pm}(x,t,k)=
\left(\tilde{M}_{(s)\pm}(x,t)+O(k+\omega_{n-s})\right)(k+\omega_{n-s})^{\sigma_3},\,
k\to-\omega_{n-s},\,s=\overline{0,m-1},
\end{equation}
where $\det\tilde{M}_{(s)\pm}(x,t)=1$ for all $x,t$.

\subsection{The RH problem deformations}\label{RHdeform}
In order to turn oscillations to exponential decay in the Riemann-Hilbert problem 
(\ref{tildeM})--(\ref{p-res-2}), we ``deform'' the contour off the real axis. When doing it, we assume that the reflection coefficients $r_j(k)$, $j=1,2$, can be analytically continued 
into the whole complex plane. 
 This takes place, for example, if $q_0(x)$ is a local perturbation of
 the pure step initial data
(\ref{shifted-step})). Alternatively,  it is possible to approximate $r_j(k)$ and $\frac{r_j(k)}{1+r_1(k)r_2(k)}$ by some rational functions with well-controlled errors (see \cite{DIZ}).

Define $\hat{M}(x,t,k)$ as follows 
(see Figure \ref{F1}; notice that all zeros of $a_1(k)$ are located in $\hat{\Omega}_0$):
\begin{figure}[h]
\centering{\includegraphics[width=0.99\linewidth]{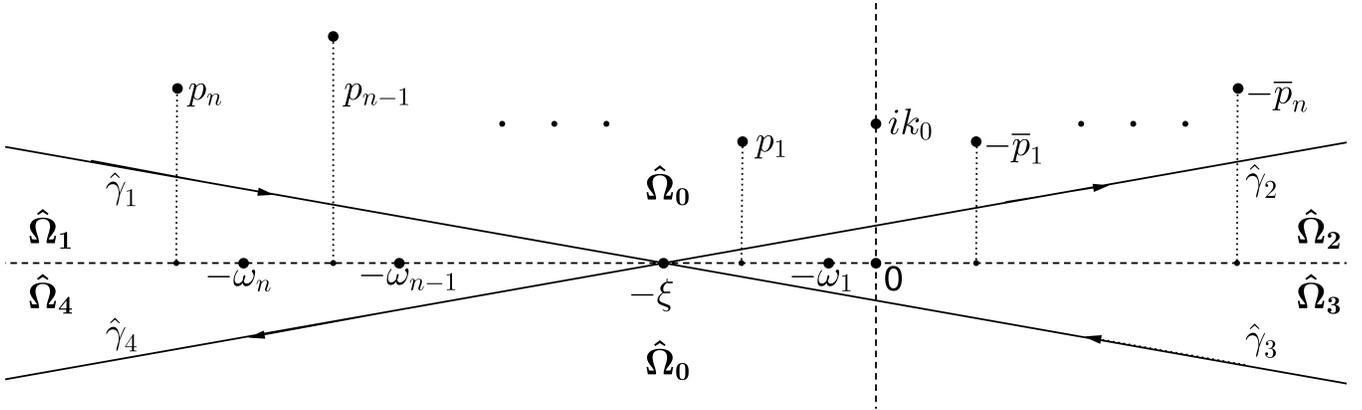}}
\caption{The domains $\hat\Omega_j$, $j=0,\dots,4$  and the contour 
$\hat\Gamma=\hat\gamma_1\cup...\cup\hat\gamma_4$}
\label{F1}
\end{figure}
\begin{equation}
\hat{M}(x,t,k)=
\begin{cases}
\tilde{M}(x,t,k),& k\in\hat\Omega_0,\\
\tilde{M}(x,t,k)
\begin{pmatrix}
1& \frac{-r_2(k)\delta^{2}(k,\xi)}{1+r_1(k)r_2(k)}e^{-2it\theta}\\
0& 1\\
\end{pmatrix}
,& k\in\hat\Omega_1,
\\
\tilde{M}(x,t,k)
\begin{pmatrix}
1& 0\\
-r_1(k)\delta^{-2}(k,\xi)e^{2it\theta}& 1\\
\end{pmatrix}
,& k\in\hat\Omega_2,
\\
\tilde{M}(x,t,k)
\begin{pmatrix}
1& r_2(k)\delta^2(k,\xi)e^{-2it\theta}\\
0& 1\\
\end{pmatrix}
,& k\in\hat\Omega_3,
\\
\tilde{M}(x,t,k)
\begin{pmatrix}
1& 0\\
\frac{r_1(k)\delta^{-2}(k,\xi)}{1+r_1(k)r_2(k)}e^{2it\theta}& 1\\
\end{pmatrix}
,& k\in\hat\Omega_4.
\end{cases}
\end{equation}
Then $\hat M(x,t,k)$ satisfies the RH problem with the jump across $\hat\Gamma$:
\begin{subequations}\label{RHhat}
\begin{align}
&\hat{M}_+(x,t,k)=\hat{M}_-(x,t,k)\hat{J}(x,t,k),&&k\in\hat\Gamma,\\
&\hat{M}(x,t,k)\rightarrow I,&&k\rightarrow\infty,
\end{align}
\end{subequations}
where the jump matrix has the form
\begin{equation}
\label{J-hat}
\hat{J}(x,t,k)=
\begin{cases}
\begin{pmatrix}
1& \frac{r_2(k)\delta^{2}(k,\xi)}{1+r_1(k)r_2(k)}e^{-2it\theta}\\
0& 1\\
\end{pmatrix}
,& k\in\hat\gamma_1,
\\
\begin{pmatrix}
1& 0\\
r_1(k)\delta^{-2}(k,\xi)e^{2it\theta}& 1\\
\end{pmatrix}
,& k\in\hat\gamma_2,
\\
\begin{pmatrix}
1& -r_2(k)\delta^2(k,\xi)e^{-2it\theta}\\
0& 1\\
\end{pmatrix}
,& k\in\hat\gamma_3,
\\
\begin{pmatrix}
1& 0\\
\frac{-r_1(k)\delta^{-2}(k,\xi)}{1+r_1(k)r_2(k)}e^{2it\theta}& 1\\
\end{pmatrix}
,& k\in\hat\gamma_4,
\end{cases}
\end{equation}

as well as with the residue conditions
\begin{subequations}
\begin{align}
&\underset{k=ik_0}{\operatorname{Res}}\hat M^{(1)}(x,t,k)=c_1(x,t)
\hat M^{(2)}(x,t,ik_0),\quad|\gamma_0|=1,\\
\label{respj}
&\underset{k=p_j}{\operatorname{Res}}\hat M^{(1)}(x,t,k)=f_j(x,t)\hat M^{(2)}(x,t,p_j),\quad j=\overline{1,n},\\
&\underset{k=-\overline{p}_j}{\operatorname{Res}}\hat M^{(1)}(x,t,k)=
\hat f_j(x,t)\hat M^{(2)}(x,t,-\overline{p}_j),\quad j=\overline{1,n},
\end{align}
\end{subequations}
where 
\begin{equation}\label{f-j}
c_1(x,t)=
\frac{\gamma_0e^{-2k_0x-4ik_0^2t}}{\dot{a}_1(ik_0)\delta^{2}(ik_0,\xi)}, \quad
f_j(x,t)=\frac{\eta_je^{2ip_jx+4ip_j^2t}}{\dot a_1(p_j)\delta^{2}(p_j,\xi)},\quad 
\hat f_j(x,t)=\frac{e^{-2i\overline{p}_jx+4i\overline{p}_j^2t}}
{\bar\eta_j\dot a_1(-\overline{p}_j)
\delta^{2}(-\overline{p}_j,\xi)},
\end{equation}
and the residue condition at $k=0$
\begin{equation}\label{zc}
\underset{k=0}{\operatorname{Res}}\hat M^{(2)}(x,t,k)=c_0(\xi)
\hat M^{(1)}(x,t,0),
\end{equation}
with $c_0(\xi)=\frac{A\delta^2(0,\xi)}{2i}$. 
Notice that it is the pseudo-residue conditions (\ref{z}) that reduce to a conventional 
residue condition (\ref{zc}). Finally, we describe
the behavior at $\{-\omega_{n-s}\}_{s=0}^{m-1}$ for $-\omega_{n-m+1}<-\xi<-\omega_{n-m}$, $m=\overline{0,n}$:
\begin{equation}\label{omegahat}
\hat{M}(x,t,k)=
\left(\hat{M}_{(s)}(x,t)+O(k+\omega_{n-s})\right)
(k+\omega_{n-s})^{\sigma_3},\,
k\to-\omega_{n-s},\,s=\overline{0,m-1},
\end{equation}
where $\hat{M}_{(s)}(x,t)$ are some (not prescribed) 
matrix functions with $\det\hat{M}_{(s)}(x,t)=1$ for all $x,t$.

\begin{proposition}\label{asRH}
For any fixed $\xi=\frac{x}{4t}$ with $\xi>0$, the solution of the Riemann-Hilbert problem 
(\ref{RHhat})-(\ref{omegahat}) can be ``reduced'', as $t\to\infty$,
 to a sectionally meromorphic matrix-valued function $M^{as}(\xi,t,k)$,
in the sense that $q(x,t)$ extracted from the large-$k$ asymptotics of 
$\hat M(x,t,k)$ and $M^{as}(\xi,t,k)$ are exponentially close as $t\to\infty$:
\begin{align}\label{solas}
&q(x,t)=2i\lim_{k\to\infty}kM_{12}^{as}(\xi,t,k)+
\mbox{exponentially small terms},\quad t\to\infty,\\
\label{sol1as}
&q(-x,t)=-2i\lim_{k\to\infty}k\overline{M_{21}^{as}(\xi,t,k)}+
\mbox{exponentially small terms},\quad t\to\infty.
\end{align}
Here $M^{as}$ solves one of the following Riemann-Hilbert problems, 
depending on the value of $\xi$, with a single residue condition
at $k=0$ (to simplify the notations, we set $\Re p_0:=0$ and $\prod\limits_{s=m_1}^{m_2}F_s=1$ if $m_1>m_2$):
\begin{enumerate}[(i)]
\item for $-\omega_{n-m+1}<-\xi<\Re p_{n-m}$, $m=\overline{0,n}$, $M^{as}$ solves
\begin{subequations}\label{RHas1}
\begin{align}
&M^{as}_+(\xi,t,k)=M^{as}_-(\xi,t,k)J^{as}(\xi,t,k),&& k\in\hat\Gamma,\\
&M^{as}(\xi,t,k)\to I, && k\to\infty,\\
\label{RHas1c}
&\underset{k=0}{\operatorname{Res}}M^{as\,(2)}(\xi,t,k)=
c_{0}^{as}(\xi)M^{as\,(1)}(\xi,t,0),
\end{align}
\end{subequations}
with 
\begin{equation}\label{c0as}
c_{0}^{as}(\xi)=\frac{A\delta^2(0,\xi)}{2i}
\prod\limits_{s=0}^{m-1}\left(\frac{\omega_{n-s}}{p_{n-s}}\right)^2,
\end{equation}
and 
$$
J^{as}(\xi,t,k)=
\left(\prod\limits_{s=0}^{m-1}\frac{k+\omega_{n-s}}{k-p_{n-s}}\right)^{\sigma_3}
\hat{J}(x,t,k)
\left(\prod\limits_{s=0}^{m-1}\frac{k+\omega_{n-s}}{k-p_{n-s}}\right)
^{-\sigma_3},\,k\in\hat\Gamma.
$$
\item for $\Re p_{n-m}<-\xi<-\omega_{n-m}$, $m=\overline{0,n-1}$, $M^{as}$ solves
\begin{subequations}\label{RHas2}
\begin{align}
&M^{as}_+(\xi,t,k)=M^{as}_-(\xi,t,k)J^{as}(\xi,t,k),&& k\in\hat\Gamma,\\
&M^{as}(\xi,t,k)\to I, && k\to\infty,\\
&\underset{k=0}{\operatorname{Res}}\ M^{as\,(1)}(\xi,t,k)=
c_{0}^{as\#}(\xi)M^{as\,(2)}(\xi,t,0),
\end{align}
\end{subequations}
with 
\begin{equation}\label{c0as2}
c_0^{as\#}(\xi)=\frac{2ip_{n-m}^2}{A\delta^2(0,\xi)}
\prod\limits_{s=0}^{m-1}\left(\frac{p_{n-s}}{\omega_{n-s}}\right)^2,
\end{equation}
and
$$
J^{as}(\xi,t,k)=
\left(d(k)
\prod\limits_{s=0}^{m-1}\frac{k+\omega_{n-s}}{k-p_{n-s}}\right)^{\sigma_3}
\hat{J}(x,t,k)
\left(d(k)
\prod\limits_{s=0}^{m-1}\frac{k+\omega_{n-s}}{k-p_{n-s}}\right)^{-\sigma_3},
\,k\in\hat\Gamma,
$$
with $d(k)=\frac{k}{k-p_{n-m}}$.
\end{enumerate}
\end{proposition}
\begin{proof}
(i) First, consider $-\omega_{n-m+1}<-\xi<\Re p_{n-m}$, $m=\overline{0,n}$.
Then $\hat M(x,t,k)$ has $m$ singular points $k=-\omega_{n-s}$, $s=\overline{0,m-1}$, and $m$ exponentially growing residue conditions at  $k=p_{n-s}$, $s=\overline{0,m-1}$, 
see (\ref{respj}). Introducing $\check M(x,t,k)$ by
\begin{equation}\label{singul}
\check M(x,t,k):=\hat M(x,t,k) 
\left(
\prod\limits_{s=0}^{m-1}\frac{k+\omega_{n-s}}{k-p_{n-s}}
\right)^{-\sigma_3},\quad k\in\mathbb{C},
\end{equation}
we obtain that 
 $\check M(x,t,k)$  is, on one hand, bounded at $k=-\omega_{n-s}$, $s=\overline{0,m-1}$, and on the other hand,  has exponentially decaying residue conditions at all points of the discrete spectrum: $ik_0$, $p_j$, $-\overline{p}_j$, $j=\overline{1,n}$. Direct calculations show that $\check M(x,t,k)$ has the residue condition at  $k=0$ and the jump  across
 $\hat\Gamma$ as indicated in (\ref{RHas1}) and (\ref{c0as}), and thus 
$q(x,t)$ as well as $q(-x,t)$ obtained via (\ref{sol}) and (\ref{sol1}) from the large-$k$ asymptotics of $\check M(x,t,k)$ 
are close, as $t\to\infty$, to that obtained from $M^{as}$ determined as the solution of the RH problem (\ref{RHas1}).

(ii) Now consider $\Re p_{n-m}<-\xi<-\omega_{n-m}$, $m=\overline{0,n-1}$.
Then  $\hat M(x,t,k)$ has $m$ singular points at $k=-\omega_{n-s}$, $s=\overline{0,m-1}$, 
and $m+1$ exponentially growing residue conditions at  $k=p_{n-s}$, $s=\overline{0,m}$. Applying  the same transformation (\ref{singul}) and ignoring the decaying 
residue conditions, we arrive at  
 the  Riemann-Hilbert problem with the exponentially growing 
residue condition at  $k=p_{n-m}$ (see (\ref{f-j})):
\begin{subequations}\label{RHaux}
\begin{align}
&\tilde M^{as}_+(\xi,t,k)=\tilde M^{as}_-(\xi,t,k)\tilde J^{as}(\xi,t,k),&& k\in\hat\Gamma,\\
&\tilde M^{as}(\xi,t,k)\to I, && k\to\infty,\\
&\underset{k=p_{n-m}}{\operatorname{Res}}\tilde M^{as\,(1)}(\xi,t,k)=
f(x,t)\tilde M^{as\,(2)}(\xi,t,p_{n-m}),\\
&\underset{k=0}{\operatorname{Res}}\tilde M^{as\,(2)}(\xi,t,k)=
c_{0}^{as}(\xi)\tilde M^{as\,(1)}(\xi,t,0),
\end{align}
\end{subequations}
where $f(x,t)=f_{n-m}(x,t)\prod\limits_{s=0}^{m-1}
\left(\frac{p_{n-m}-p_{n-s}}{p_{n-m}+\omega_{n-s}}\right)^2$, 
$c_0^{as}(\xi)$ is given by (\ref{c0as}), and
$$
\tilde J^{as}(x,t,k)=
\left(
\prod\limits_{s=0}^{m-1}\frac{k+\omega_{n-s}}{k-p_{n-s}}
\right)^{\sigma_3}
\hat{J}(x,t,k)
\left(
\prod\limits_{s=0}^{m-1}\frac{k+\omega_{n-s}}{k-p_{n-s}}
\right)^{-\sigma_3},\,k\in\hat \Gamma.
$$

In order to cope with the problem of growing residue condition, first we reformulate 
the RH problem (\ref{RHaux}) in such a way that instead of the residue conditions we will have appropriate jumps across small (counterclockwise oriented) circles $S_0$ and $S_{p_{n-m}}$ centered
 at  $k=0$ and $k=p_{n-m}$ respectively:
\begin{equation}
\nonumber
\hat M^{as}(x,t,k)=
\begin{cases}
\tilde{M}^{as}(x,t,k)
\begin{pmatrix}
1& -\frac{c_0^{as}(\xi)}{k}\\
0& 1\\
\end{pmatrix},
& k\mbox{ inside } S_0,\\
\tilde{M}^{as}(x,t,k)
\begin{pmatrix}
1& 0\\
-\frac{f(x,t)}{k-p_{n-m}}& 1\\
\end{pmatrix},
& k\mbox{ inside } S_{p_{n-m}},\\
\tilde{M}^{as}(x,t,k),
& \mbox{ otherwise }.
\end{cases}
\end{equation}
Then $\hat M^{as}(x,t,k)$ solves the following Riemann-Hilbert problem:
\begin{subequations}\label{RHaux2}
\begin{align}
&\hat M^{as}_+(x,t,k)=\hat M^{as}_-(x,t,k)\hat J^{as}(x,t,k),&& k\in\hat\Gamma\cup S_0\cup S_{p_{n-m}},\\
&\hat M^{as}(x,t,k)\to I, && k\to\infty,
\end{align}
\end{subequations}
with
\begin{equation}
\hat J^{as}(x,t,k)=
\begin{cases}
\tilde J^{as}(x,t,k),& k\in\hat\Gamma,\\
\begin{pmatrix}
1& -\frac{c_0^{as}(\xi)}{k}\\
0& 1
\end{pmatrix}, & k\in S_0,\\
\begin{pmatrix}
1& 0\\
-\frac{f(x,t)}{k-p_{n-m}}& 1
\end{pmatrix},& k\in S_{p_{n-m}}.
\end{cases}
\end{equation}

Now we introduce $\hat M^{as\#}(x,t,k)$ as follows:
\begin{equation}
\label{35}
\hat M^{as\#}(x,t,k)=
\begin{cases}
\hat{M}^{as}(x,t,k)N(\xi,k)
d^{-\sigma_3}(k),
& k\mbox{ inside } S_0,\\
\hat{M}^{as}(x,t,k)Q(x,t,k)
d^{-\sigma_3}(k),
& k\mbox{ inside } S_{p_{n-m}},\\
\hat{M}^{as}(x,t,k)
d^{-\sigma_3}(k),
& \mbox{ otherwise },
\end{cases}
\end{equation}
where $d(k)=\frac{k}{k-p_{n-m}}$,
$
N(\xi,k)=
\begin{pmatrix}
0& \frac{c_0^{as}(\xi)}{k}\\
-\frac{k}{c_0^{as}(\xi)}& 1
\end{pmatrix}
$, and
$
Q(x,t,k)=
\begin{pmatrix}
1& -\frac{k-p_{n-m}}{f(x,t)}\\
\frac{f(x,t)}{k-p_{n-m}}& 0
\end{pmatrix},
$
and notice  that $\hat M^{as\#}(x,t,k)$ solves the  Riemann-Hilbert problem
with decaying (to $I$) jump matrix across $S_{p_{n-m}}$:
\begin{subequations}\label{RHaux3}
\begin{align}
&\hat M^{as\#}_+(x,t,k)=\hat M^{as\#}_-(x,t,k)\hat J^{as\#}(x,t,k),&& k\in\hat\Gamma\cup S_0\cup S_{p_{n-m}},\\
&\hat M^{as\#}(x,t,k)\to I, && k\to\infty,
\end{align}
\end{subequations}
where
\begin{equation}
\hat J^{as\#}(x,t,k)=
\begin{cases}
d^{\sigma_3}(k)
\hat J^{as}(x,t,k)
d^{-\sigma_3}(k),& k\in\hat\Gamma,\\
\begin{pmatrix}
1& 0\\
-\frac{(k-p_{n-m})^2}{c_0^{as}(\xi)k}& 1
\end{pmatrix}, & k \text{ inside }  S_0,\\
\begin{pmatrix}
1& -\frac{k^2}{f(x,t)(k-p_{n-m})}\\
0& 1
\end{pmatrix},& k \text{ inside }  S_{p_{n-m}}.
\end{cases}
\end{equation}
Finally, introducing 
\begin{equation}
\check M^{as\#}(x,t,k)=
\begin{cases}
\hat M^{as\#}(x,t,k)
\begin{pmatrix}
1& 0\\
\frac{k-2p_{n-m}}{c_0^{as}(\xi)}& 1\\
\end{pmatrix},
& k\mbox{ inside } S_0,\\
\hat M^{as\#}(x,t,k),& \mbox{ otherwise },
\end{cases}
\end{equation}
and ignoring the decaying jump across $S_{p_{n-m}}$, we arrive at the Riemann-Hilbert problem (\ref{RHas2}) and, as in the case (i), $q(x,t)$ and $q(-x,t)$ obtained from (\ref{sol}) and (\ref{sol1}) are exponentially close to that obtained from $M^{as}$.
\end{proof}

\begin{corollary}
As $t\to\infty$, the solution has the following asymptotics (see also Figure \ref{fas})
\begin{equation}\label{asq1}
q(x,t)=\left\{
\begin{aligned}
& A\delta^2(0,\xi)
\prod\limits_{s=0}^{m-1}\left(\frac{\omega_{n-s}}{p_{n-s}}\right)^2+o(1),& -\Re p_{n-m}<\xi<\omega_{n-m+1},\\
&o(1),& -\omega_{n-m+1}<\xi<\Re p_{n-m}\ \text{and}\ \omega_{n-m}<\xi<-\Re p_{n-m},\\
&\frac{-4\overline{p}_{n-m}^2}{A\overline{\delta^2}(0,-\xi)}
\prod\limits_{s=0}^{m-1}\left(\frac{\overline{p}_{n-s}}
{\overline{\omega}_{n-s}}\right)^2+o(1),& \Re p_{n-m}<\xi<-\omega_{n-m},\\
\end{aligned}
\right.
\end{equation}
where $m=\overline{0,n}$.
\end{corollary}
\begin{figure}[h]
\begin{minipage}[h]{0.5\linewidth}
\centering{\includegraphics[width=0.99\linewidth]{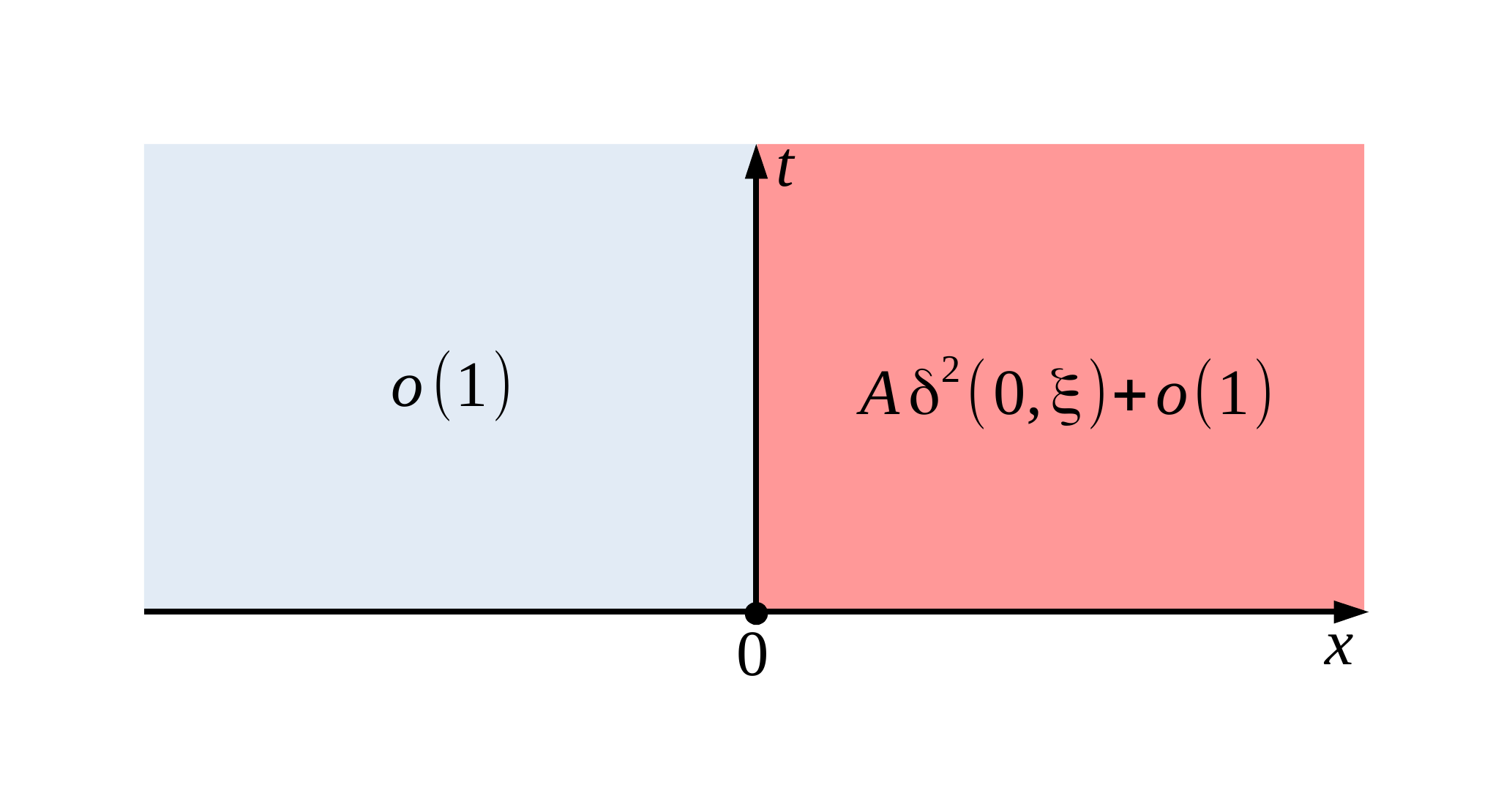}}
\end{minipage}
\hfill
\begin{minipage}[h]{0.5\linewidth}
\centering{\includegraphics[width=0.99\linewidth]{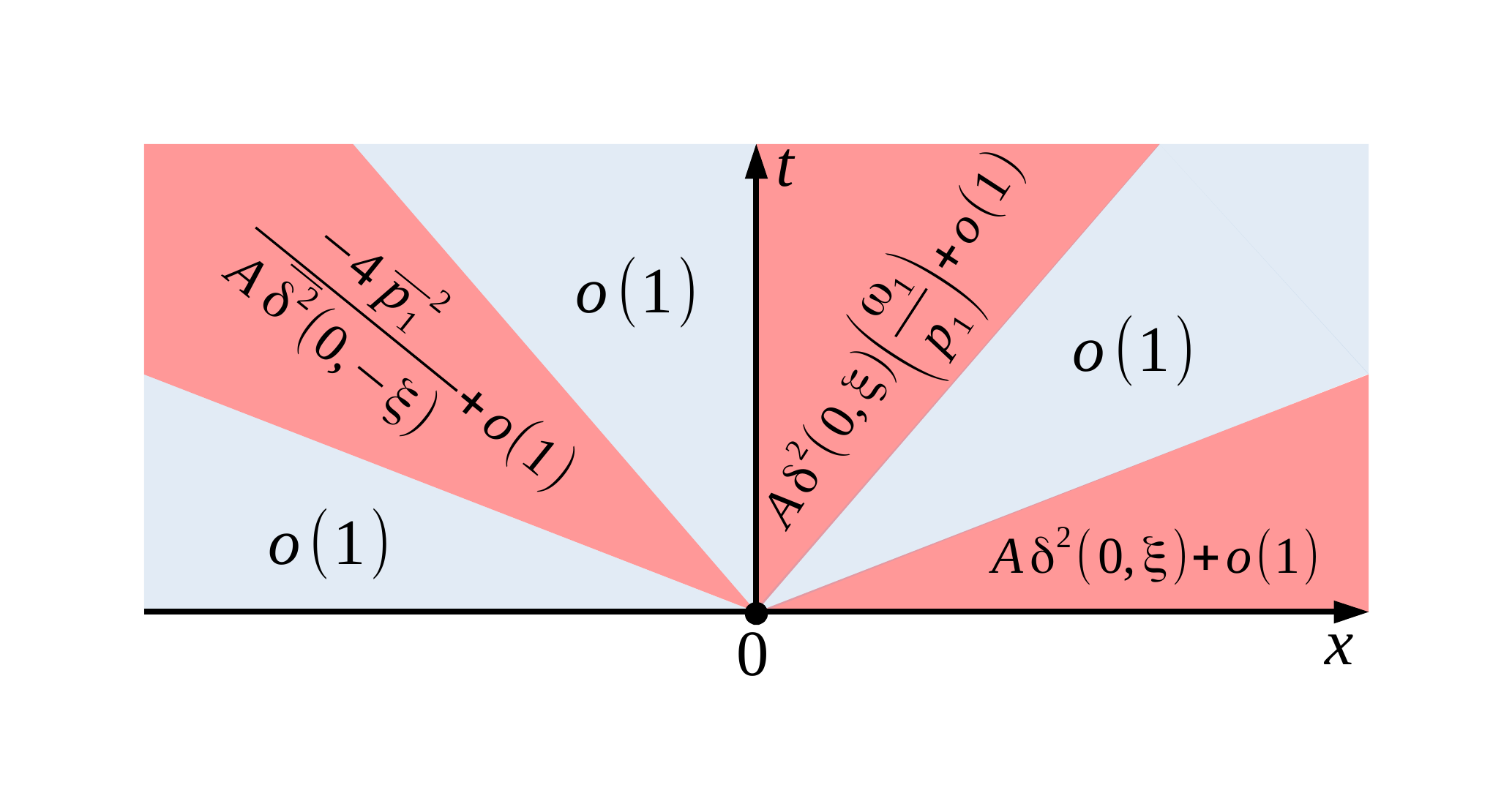}}
\end{minipage}
\caption{Asymptotic behavior of the solution for $n=0$ and $n=1$.}
\label{fas}
\end{figure}
\begin{remark}
The asymptotic formula (\ref{asq1}) holds in the case of the ``shifted step'' initial value (\ref{shifted-step}) with $n=0$ for $0<R<\frac{\pi}{2A}$, and with the corresponding value of 
$n\in\mathbb{N}$ for 
$\frac{(2n-1)\pi}{2A}<R<\frac{(2n+1)\pi}{2A}$. 
\end{remark}

\begin{remark}
Notice, that the spectral functions associated to the ``shifted step'' initial value (\ref{shifted-step}) with $0<R<\frac{\pi}{2A}$ satisfy conditions of the Theorem 1 in \cite{RS2}. This rough asymptotics, as well as the precise one (see Theorem \ref{th1} below) are consistent with that obtained in \cite{RS2}.
\end{remark}

\begin{remark}
The ordering  of $\Re p_j$ and $-\omega_j$, $j=\overline{1,n}$ in (\ref{order}) is crucial for our analysis. Indeed, let  $n=1$ and assume that $-\omega_1<\Re p_1<0$. Then, applying  (\ref{singul}) for $-\omega_1<-\xi<\Re p_1$, we (asymptotically) arrive at the following Riemann-Hilbert problem:
\begin{subequations}
\begin{align}
&\tilde{M}^{as}_+(x,t,k)=\tilde{M}^{as}_-(x,t,k)\tilde{J}^{as}(x,t,k),&& k\in \hat\Gamma,\\
&\tilde{M}^{as}(x,t,k)\to I,&& k\to\infty,\\
\label{res1}
&\underset{k=p_{1}}{\operatorname{Res}}\tilde M^{as\,(2)}(x,t,k)=
f_{1}^{-1}(x,t)(p_1+\omega_1)^2\tilde M^{as\,(1)}(x,t,p_{1}),\\
\label{res2}
&\underset{k=0}{\operatorname{Res}}\tilde M^{as\,(2)}(x,t,k)=
c_{0}\frac{\omega_1^2}{p_1^2}\tilde M^{as\,(1)}(x,t,0),
\end{align}
\end{subequations}
where 
$
\tilde J^{as}(x,t,k)=
\left(\frac{k+\omega_{1}}{k-p_{1}}\right)^{\sigma_3}
\hat{J}(x,t,k)
\left(\frac{k+\omega_{1}}{k-p_{1}}\right)^{-\sigma_3},\,k\in\hat \Gamma,
$ and $f_1^{-1}(x,t)$ is exponentially growing. Since the residue conditions (\ref{res1}) and (\ref{res2}) are formulated  for the same column,
 we cannot proceed  as in the proof of Proposition \ref{asRH} above.
\end{remark}

Applying the nonlinear steepest descent method \cite{DIZ,DZ} we are able to make the
asymptotics presented in (\ref{asq1}) more precise.

\begin{theorem}\label{th1}
Consider the Cauchy problem (\ref{1}) and assume that the initial value $q_0(x)$ converges to its boundary values fast enough and that associated spectral functions $a_j(k)$, $j=1,2$ satisfy Assumptions A.
Assuming that the solution $q(x,t)$ of (\ref{1}) exists, it has the following long-time asymptotics (for convenience of notation we set $\Re p_0 := 0$ and $\prod\limits_{s=m_1}^{m_2}F_s=1$ if $m_1>m_2$):
\begin{description}
\item{\textbf{(i)}} for $-\omega_{n-m+1}<-\xi<\Re p_{n-m}$, $m=\overline{0,n}$ we have three types of asymptotics, depending on the value of $\Im\nu(-\xi)$:
\begin{enumerate}[1)]
\item if $\Im\nu(-\xi)\in\left(-\frac{1}{2}, -\frac{1}{6}\right]$, then
\begin{equation*}
q(x,t)=A\delta^2(0,\xi)
\prod\limits_{s=0}^{m-1}\left(\frac{\omega_{n-s}}{p_{n-s}}\right)^2
+t^{-\frac{1}{2}-\Im\nu(-\xi)}\alpha_1(\xi)\exp\{-4it\xi^2+i\Re\nu(-\xi)\ln t\}
+ R_1(\xi,t).
\end{equation*}
\item if $\Im\nu(-\xi)\in\left(-\frac{1}{6}, \frac{1}{6}\right)$, then
\begin{align}
\nonumber
q(x,t)=&A\delta^2(0,\xi)
\prod\limits_{s=0}^{m-1}\left(\frac{\omega_{n-s}}{p_{n-s}}\right)^2
+t^{-\frac{1}{2}-\Im\nu(-\xi)}\alpha_1(\xi)\exp\{-4it\xi^2+i\Re\nu(-\xi)
\ln t\}\\
\nonumber
&+t^{-\frac{1}{2}+\Im\nu(-\xi)}\alpha_2(\xi)\exp\{4it\xi^2-i\Re\nu(-\xi)
\ln t\}+R_3(\xi,t).
\end{align}
\item if $\Im\nu(-\xi)\in\left[\frac{1}{6}, \frac{1}{2}\right)$, then
\begin{equation*}
q(x,t)=A\delta^2(0,\xi)
\prod\limits_{s=0}^{m-1}\left(\frac{\omega_{n-s}}{p_{n-s}}\right)^2
+t^{-\frac{1}{2}+\Im\nu(-\xi)}\alpha_2(\xi)\exp\{4it\xi^2-i\Re\nu(-\xi)
\ln t\}+R_2(\xi,t).
\end{equation*}
\end{enumerate}
\item{\textbf{(ii)}} for $-\Re p_{n-m}<-\xi<\omega_{n-m+1}$, $m=\overline{0,n}$:
\begin{equation*}
q(x,t)=t^{-\frac{1}{2}-\Im\nu(\xi)}\alpha_3(\xi)
\exp\{4it\xi^2-i\Re\nu(\xi)\ln t\}
+ R_2(-\xi,t)
\end{equation*}
\item{\textbf{(iii)}} for $\Re p_{n-m}<-\xi<-\omega_{n-m}$, $m=\overline{0,n-1}$:
\begin{equation*}
q(x,t)=t^{-\frac{1}{2}+\Im\nu(-\xi)}\alpha_4(\xi)
\exp\{4it\xi^2-i\Re\nu(-\xi)\ln t\}
+ R_2(\xi,t)
\end{equation*}
\item{\textbf{(iv)}} for $\omega_{n-m}<-\xi<-\Re p_{n-m}$, $m=\overline{0,n-1}$ we have three types of asymptotics, depending on the value of $\Im\nu(\xi)$:
\begin{enumerate}[1)]
\item if $\Im\nu(\xi)\in\left(-\frac{1}{2}, -\frac{1}{6}\right]$, then
\begin{equation*}
q(x,t)=\frac{-4\overline{p}_{n-m}^2}{A\overline{\delta^2}(0,-\xi)}
\prod\limits_{s=0}^{m-1}\left(\frac{\overline{p}_{n-s}}
{\overline{\omega}_{n-s}}\right)^2
+t^{-\frac{1}{2}-\Im\nu(\xi)}\alpha_5(\xi)\exp\{4it\xi^2-i\Re\nu(\xi)
\ln t\}+R_1(-\xi,t),
\end{equation*}
\item if $\Im\nu(\xi)\in\left(-\frac{1}{6}, \frac{1}{6}\right)$, then
\begin{align}
\nonumber
q(x,t)=&\frac{-4\overline{p}_{n-m}^2}{A\overline{\delta^2}(0,-\xi)}
\prod\limits_{s=0}^{m-1}\left(\frac{\overline{p}_{n-s}}
{\overline{\omega}_{n-s}}\right)^2
+t^{-\frac{1}{2}-\Im\nu(\xi)}\alpha_5(\xi)\exp\{4it\xi^2-i\Re\nu(\xi)
\ln t\}\\
\nonumber
&+t^{-\frac{1}{2}+\Im\nu(\xi)}\alpha_6(\xi)\exp\{-4it\xi^2+i\Re\nu(\xi)
\ln t\}+R_3(-\xi,t).
\end{align}
\item if $\Im\nu(
\xi)\in\left[\frac{1}{6}, \frac{1}{2}\right)$, then
\begin{equation*}
q(x,t)=\frac{-4\overline{p}_{n-m}^2}{A\overline{\delta^2}(0,-\xi)}
\prod\limits_{s=0}^{m-1}\left(\frac{\overline{p}_{n-s}}
{\overline{\omega}_{n-s}}\right)^2
+t^{-\frac{1}{2}+\Im\nu(\xi)}\alpha_6(\xi)\exp\{-4it\xi^2+i\Re\nu(\xi)\ln t\}
+ R_2(-\xi,t).
\end{equation*}
\end{enumerate}
\end{description}
Here
\begin{equation}
\delta(k,\xi)=
(k+\xi)^{i\nu(-\xi)}\prod\limits_{s=0}^{m-1}(k+\omega_{n-s})^{-1}
\exp\left\{\sum\limits_{s=0}^{m}\chi_s(k)\right\},
\end{equation}
and
\begin{equation}
\nu(-\xi)=-\frac{1}{2\pi}\ln|1+r_1(-\xi)r_2(-\xi)|
-\frac{i}{2\pi}\left(\int_{-\infty}^{-\xi}\,
d\arg(1+r_1(\zeta)r_2(\zeta))+2\pi m\right),
\end{equation}
with
\begin{subequations}
	\begin{align}
	&\chi_s(k)=-\frac{1}{2\pi i}\int_{-\omega_{n-s+1}}^{-\omega_{n-s}}
	\ln(k-\zeta)\,d_{\zeta}\ln(1+r_1(\zeta)r_2(\zeta)),\quad s=\overline{0,m-1},\\
	&\chi_m(k)=-\frac{1}{2\pi i}\int_{-\omega_{n-m+1}}^{-\xi}
	\ln(k-\zeta)\,d_{\zeta}\ln(1+r_1(\zeta)r_2(\zeta)).
	\end{align}
\end{subequations}
The constants $\alpha_j(\xi)$, $j=\overline{1,6}$ are as follows:
$$
\alpha_1(\xi)=
\frac{\sqrt{\pi}(c_0^{as}(\xi))^2\prod\limits_{s=0}^{m-1}(\xi+p_{n-s})^2}
{\xi^2 r_2(-\xi)\Gamma(i\nu(-\xi))}
\exp\left\{-\frac{\pi}{2}\nu(-\xi)+\frac{3\pi i}{4}-2\sum\limits_{s=0}^{m}
\chi_s(-\xi)+3i\nu(-\xi)\ln 2\right\},
$$
$$
\alpha_2(\xi)=
\frac{\sqrt{\pi}\prod\limits_{s=0}^{m-1}(\xi+p_{n-s})^{-2}}
{r_1(-\xi)\Gamma(-i\nu(-\xi))}
\exp\left\{-\frac{\pi}{2}\nu(-\xi)+\frac{\pi i}{4}+2\sum\limits_{s=0}^{m}
\chi_s(-\xi)-3i\nu(-\xi)\ln 2\right\},
$$
$$
\alpha_3(\xi)=
\frac{\sqrt{\pi}\prod\limits_{s=0}^{m-1}(\overline{p}_{n-s}-\xi)^{2}}
{\overline{r_2}(\xi)\Gamma(-i\overline{\nu(\xi)})}
\exp\left\{-\frac{\pi}{2}\overline{\nu(\xi)}+\frac{\pi i}{4}-2\sum\limits_{s=0}^{m}
\overline{\chi_s(\xi)}-3i\overline{\nu(\xi)}\ln 2\right\},
$$
$$
\alpha_4(\xi)=
\frac{\sqrt{\pi}\xi^2\prod\limits_{s=0}^{m}(\xi+p_{n-s})^{-2}}
{r_1(-\xi)\Gamma(-i\nu(-\xi))}
\exp\left\{-\frac{\pi}{2}\nu(-\xi)+\frac{\pi i}{4}+2\sum\limits_{s=0}^{m}
\chi_s(-\xi)-3i\nu(-\xi)\ln 2\right\},
$$
$$
\alpha_5(\xi)=
\frac{\sqrt{\pi}\prod\limits_{s=0}^{m}(\overline{p}_{n-s}-\xi)^{2}}
{\xi^2\overline{r_2}(\xi)\Gamma(-i\overline{\nu(\xi)})}
\exp\left\{-\frac{\pi}{2}\overline{\nu(\xi)}+\frac{\pi i}{4}
-2\sum\limits_{s=0}^{m}\overline{\chi_s(\xi)}-3i\overline{\nu(\xi)}
\ln 2\right\},
$$
$$
\alpha_6(\xi)=
\frac{\sqrt{\pi}\left(\overline{c_0^{as\#}(-\xi)}\right)^2
\prod\limits_{s=0}^{m}(\overline{p}_{n-s}-\xi)^{-2}}
{\overline{r_1}(\xi)\Gamma(i\overline{\nu(\xi)})}
\exp\left\{-\frac{\pi}{2}\overline{\nu(\xi)}+\frac{3\pi i}{4}
+2\sum\limits_{s=0}^{m}\overline{\chi_s(\xi)}+3i\overline{\nu(\xi)}
\ln 2\right\},
$$
where 
$$
c_{0}^{as}(\xi)=\frac{A\delta^2(0,\xi)}{2i}
\prod\limits_{s=0}^{m-1}\left(\frac{\omega_{n-s}}{p_{n-s}}\right)^2,\quad
c_0^{as\#}(\xi)=\frac{2ip_{n-m}^2}{A\delta^2(0,\xi)}
\prod\limits_{s=0}^{m-1}\left(\frac{p_{n-s}}{\omega_{n-s}}\right)^2.
$$
Finally, the remainders $R_j(\xi,t)$, $j=\overline{1,3}$ are as follows:
\begin{equation}
\label{R1}
R_1(\xi,t)=
\begin{cases}
O\left(t^{-1}\right),& \Im\nu(-\xi)>0,\\
O\left(t^{-1}\ln t\right),&\Im\nu(-\xi)=0,\\
O\left(t^{-1+2|\Im\nu(-\xi)|}\right),&\Im\nu(-\xi)<0,
\end{cases}
\end{equation}
\begin{equation}
\label{R2}
R_2(\xi,t)=
\begin{cases}
O\left(t^{-1+2|\Im\nu(-\xi)|}\right),& \Im\nu(-\xi)>0,\\
O\left(t^{-1}\ln t\right),&\Im\nu(-\xi)=0,\\
O\left(t^{-1}\right),&\Im\nu(-\xi)<0,
\end{cases}
\end{equation}
and
\begin{equation*}
R_3(\xi,t)=R_1(\xi,t) + R_2(\xi,t)=
\begin{cases}
O\left(t^{-1+2|\Im\nu(-\xi)|}\right),&\Im\nu(-\xi)\not=0,\\
O\left(t^{-1}\ln t\right),&\Im\nu(-\xi)=0.
\end{cases}
\end{equation*}
\end{theorem}
\noindent\textit{Sketch of proof of Theorem \ref{th1}}.
We apply the nonlinear steepest descent method to 
 the Riemann-Hilbert problems (\ref{RHas1}) and (\ref{RHas2}). The implementation of the method is close to that presented in \cite{RS}, so here we briefly describe the main steps of the proof, paying attention to its peculiarities due to Assumptions A and referring the reader to \cite{RS} for details.

We begin with the asymptotics for the Riemann-Hilbert problem (\ref{RHas1}), 
the analysis for (\ref{RHas2}) being similar (see also Remark \ref{remRHas}). 
First, we reformulate (\ref{RHas1}) in such a way that instead of the residue condition we have the jump across a small counterclockwise oriented circle $S_0$ centered at  $k=0$:
\begin{equation}
\nonumber
\check M^{as}(x,t,k)=
\begin{cases}
M^{as}(x,t,k)
\begin{pmatrix}
1& -\frac{c_0^{as}}{k}\\
0& 1\\
\end{pmatrix},
& k\mbox{ inside } S_{0},\\
M^{as}(x,t,k),
& \mbox{ otherwise }.
\end{cases}
\end{equation}
Then $\check M^{as}(x,t,k)$ solves the  Riemann-Hilbert problem
\begin{subequations}
\begin{align}
&\check M^{as}_+(x,t,k)=\check M^{as}_-(x,t,k)\check J^{as}(x,t,k),&& k\in\hat\Gamma\cup S_0,\\
&\check M^{as}(x,t,k)\to I, && k\to\infty,
\end{align}
\end{subequations}
with
\begin{equation}
\check J^{as}(x,t,k)=
\begin{cases}
J^{as}(x,t,k),& k\in\hat\Gamma,\\
\begin{pmatrix}
1& -\frac{c_0^{as}}{k}\\
0& 1
\end{pmatrix}, & k\in S_0.
\end{cases}
\end{equation}
Introduce the rescaled variable $z$ 
by
\begin{equation}\label{k-z}
k=\frac{z}{\sqrt{8t}}-\xi,
\end{equation}
so that 
\[
e^{2it\theta} = e^{\frac{iz^2}{2}-4it\xi^2}.
\]
Introduce the ``local parametrix'' $\check m^{as}_0(x,t,k)$
as the solution of a RH problem with the ``simplified'' jump matrix $J^{as}(x,t,k)$
in the sense that in its construction, $r_j(k)$, $j=1,2$ are replaced by the constants $r_j(-\xi)$
and $\delta(k,\xi;\{\omega_{n-s}\}_{s=0}^{m-1})$ is replaced by (cf. (\ref{delta}))
$$
\delta\simeq\left(\frac{z}{\sqrt{8t}}\right)^{i\nu(-\xi)}
\prod\limits_{s=0}^{m-1}(\omega_{n-s}-\xi)^{-1}
\exp\left\{\sum\limits_{s=0}^{m}\chi_s(-\xi)\right\}.
$$
Such RH problem can be solved explicitly 
in terms of the parabolic cylinder functions \cite{I1,RS}.

Indeed, $\check m^{as}_0(x,t,k)$ (cf. with $\tilde m_0(x,t,k)$ in \cite{RS}) can be determined by 
\begin{equation}\label{m0-R}
\check m^{as}_0(x,t,k)=\Delta(\xi,t)m^{\Gamma}(\xi,z(k)) \Delta^{-1}(\xi,t),
\end{equation}
where 
\begin{equation}\label{Delta}
\Delta(\xi,t) = e^{(2 i t \xi^2 + \sum\limits_{s=0}^{m}\chi_s(-\xi))\sigma_3}
\left((8t)^{\frac{i\nu(-\xi)}{2}}\prod\limits_{s=0}^{m-1}(\omega_{n-s}-\xi)\right)^{-\sigma_3},
\end{equation}
$m^{\Gamma}(\xi,z)$ is determined by 
\begin{equation}\label{m-g-0}
m^{\Gamma}(\xi,z) = m_0(\xi,z) D^{-1}_{j}(\xi,z),\qquad z\in\Omega_j,\,\,j=\overline{0,4},
\end{equation}
see Figure \ref{mod2},
where 
$$
D_0(\xi,z)=e^{-i\frac{z^2}{4}\sigma_3}z^{i\nu(-\xi)\sigma_3},
$$
\begin{equation}
\nonumber
\begin{matrix}
D_1(\xi,z)=D_0(\xi,z)
\begin{pmatrix}
1& \frac{r^{as}_2(-\xi)}{1+ r^{as}_1(-\xi)r^{as}_2(-\xi)}\\
0& 1\\
\end{pmatrix},
&&
D_2(\xi,z)=D_0(\xi,z)
\begin{pmatrix}
1& 0\\
r^{as}_1(-\xi)& 1\\
\end{pmatrix},\\
D_3(\xi,z)=D_0(\xi,z)
\begin{pmatrix}
1& - r^{as}_2(-\xi)\\
0& 1\\
\end{pmatrix},
&&
D_4(\xi,z)=D_0(\xi,z)
\begin{pmatrix}
1& 0\\
\frac{-r^{as}_1(-\xi)}{1+ r^{as}_1(-\xi)r^{as}_2(-\xi)}& 1\\
\end{pmatrix},
\end{matrix}
\end{equation}
with
\begin{equation}\label{ras}
r^{as}_1(k)=r_1(k)
\prod\limits_{s=0}^{m-1}\left(\frac{k-p_{n-s}}{k+\omega_{n-s}}\right)^2,
\quad
r^{as}_2(k)=r_2(k)
\prod\limits_{s=0}^{m-1}\left(\frac{k+\omega_{n-s}}{k-p_{n-s}}\right)^2,
\end{equation}
and $m_0(\xi,z)$ is the solution of the following RH problem in $z$-plane, 
relative to $\mathbb R$,
with a constant jump matrix:
\begin{subequations}\label{as8}
\begin{align}
&m_{0+}(\xi,z)=m_{0-}(\xi,z)j_0(\xi),&& z\in\mathbb{R},\\
&m_0(\xi,z)= \left(I+O(1/z)\right)
e^{-i\frac{z^2}{4}\sigma_3}z^{i\nu(-\xi)\sigma_3},&& z\rightarrow\infty,
\end{align}
\end{subequations}
where  
\begin{equation}\label{j0}
j_0(\xi)=
\begin{pmatrix}
1+ r^{as}_1(-\xi)r^{as}_2(-\xi) & r^{as}_2(-\xi)\\
r^{as}_1(-\xi) & 1
\end{pmatrix}.
\end{equation}

\begin{figure}[h]
	\begin{minipage}[h]{0.99\linewidth}
		\centering{\includegraphics[width=0.5\linewidth]{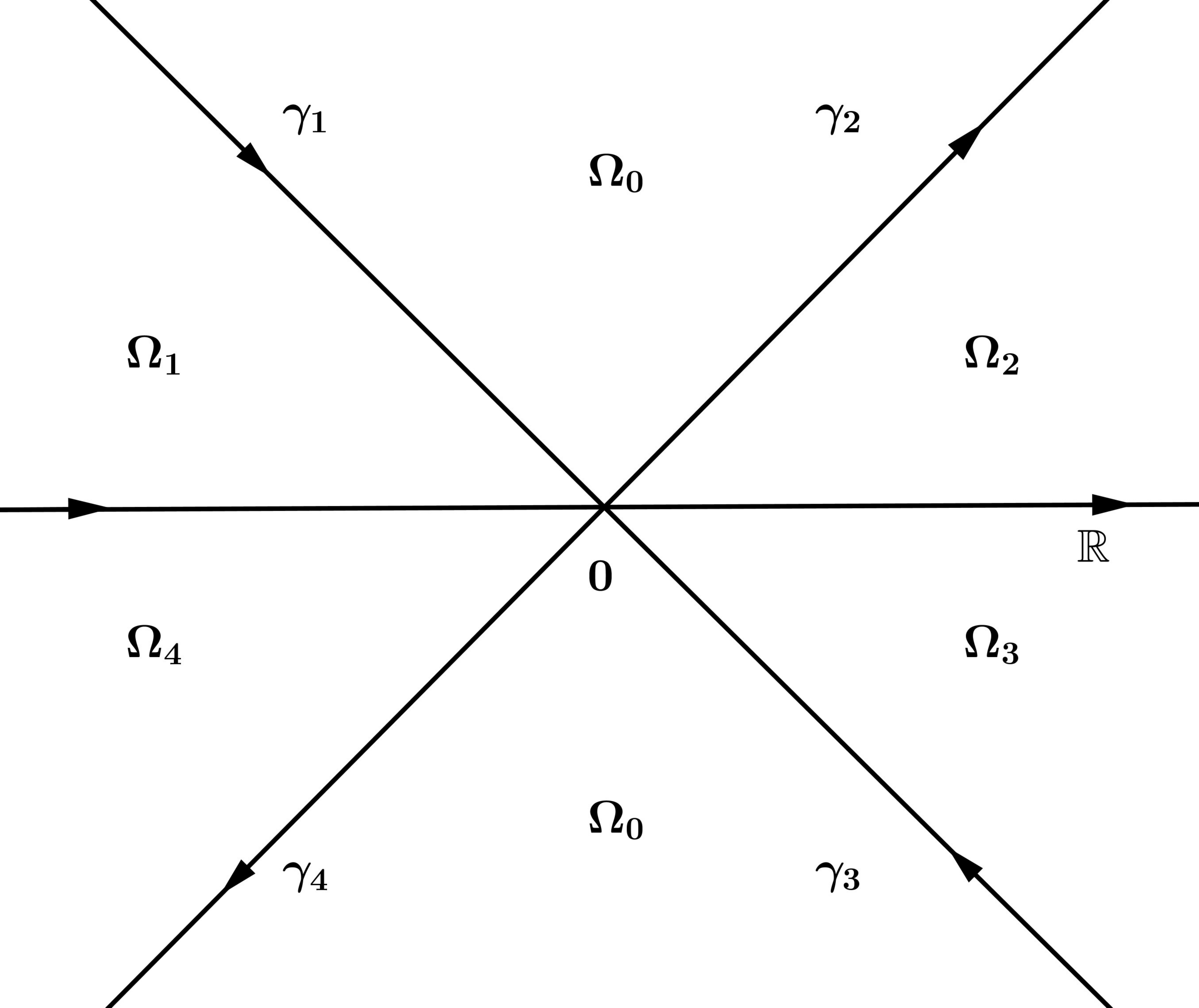}}
		\caption{Contour and domains for $m^{\Gamma}(\xi,z)$ in the $z$-plane  }
		\label{mod2}
	\end{minipage}
\end{figure}

It is the RH problem for $m_0(\xi,z)$ that can be solved explicitly, in terms of the parabolic
cylinder functions, see, e.g., Appendix A in \cite{RS}.
Since we are interested what happens for large $t$, we actually need from $m_0(\xi,z)$ (and, correspondingly, 
$m^{\Gamma}(\xi,z)$) its large-$z$ asymptotics only. The latter has  the form
\[
m^{\Gamma}(\xi,z) = I + \frac{i}{z}\begin{pmatrix}
0 & \beta(\xi) \\ -\gamma(\xi) & 0
\end{pmatrix} + O(z^{-2}), \qquad z\to \infty,
\]
where
\begin{subequations}\label{be-ga-R}
\begin{align}
\beta(\xi)=\dfrac{\sqrt{2\pi}e^{-\frac{\pi}{2}\nu(-\xi)}e^{-\frac{3\pi i}{4}}}{r^{as}_1(-\xi)
\Gamma(-i\nu(-\xi))},\\
\gamma(\xi)=\dfrac{\sqrt{2\pi}e^{-\frac{\pi}{2}\nu(-\xi)}e^{-\frac{\pi i}{4}}}{r^{as}_2(-\xi)\Gamma(i\nu(-\xi))}.
\end{align}
\end{subequations}
Now, having defined the parametrix $\check m^{as}_0(x,t,k)$,
we define $\breve M^{as}(x,t,k)$ as follows (cf.  $\hat{m}(x,t,k)$ in \cite{RS}):
\begin{equation}\label{breveM}
\breve M^{as}(x,t,k) = 
\begin{cases}
\check M^{as}(x,t,k)(\check m_0^{as})^{-1}(x,t,k)V(k), & 
k\mbox{ inside } S_{-\xi}, \\
\check M^{as}(x,t,k), & k\mbox{ inside } S_0,\\
\check M^{as}(x,t,k)V(k), & \mbox{ otherwise },
\end{cases}
\end{equation}
where $V(k)=\begin{pmatrix}1& -\frac{c_0^{as}}{k}\\0& 1 \end{pmatrix}$,
and $S_{-\xi}$ is a small counterclockwise oriented circle centered at  $k=-\xi$.
Then the sectionally analytic matrix $\breve M^{as}$ solves the following Riemann-Hilbert problem on the contour  
$\hat\Gamma_1=\hat\Gamma\cup S_{-\xi}$:
\begin{align}\label{M-breve-RHP}
&\breve M^{as}_+(x,t,k) = \breve M^{as}_-(x,t,k)\breve J^{as}(x,t,k), && k\in\hat\Gamma_1,\\
&\breve M^{as}(x,t,k)\to I, && k\to \infty,
\end{align}
with the jump matrix (cf.  (3.23) in \cite{RS})
\begin{equation}\label{J-breve}
\breve J^{as}(x,t,k) = 
\begin{cases}
V^{-1}(k)\check m_{0-}^{as}(x,t,k) \check J^{as}(x,t,k)
(\check m_{0+}^{as})^{-1}(x,t,k)
V(k), & k\in \hat\Gamma_1, k\mbox{ inside } S_{-\xi},\\
V^{-1}(k)(\check m_{0}^{as})^{-1} (x,t,k)V(k), & k\in S_{-\xi}, \\
V^{-1}(k)\check J^{as}(x,t,k)V(k), & \text{otherwise}.
\end{cases}
\end{equation}
Observe that the solution of the original problem is given in terms of 
 $\breve M^{as}(x,t,k)$ as follows:
\begin{equation}
q(x,t)=2i\left(c_0^{as}+\lim_{k\to\infty}k\breve M^{as}_{12}(x,t,k)\right)
\end{equation}
and
\begin{equation}
q(-x,t)=-2i\lim_{k\to\infty}k\overline{\breve M^{as}_{21}(x,t,k)}.
\end{equation}

Notice that $(\check m_{0}^{as})^{-1} (x,t,k)$ has the following large-$t$ asymptotics:
\begin{equation}\label{m_0}
(\check m_{0}^{as})^{-1} (x,t,k)=\Delta(\xi,t)
(m^{\Gamma})^{-1}(\xi,\sqrt{8t}(k+\xi))\Delta^{-1}(\xi,t)=
I+\frac{B(\xi,t)}{\sqrt{8t}(k+\xi)}+\tilde{r}(\xi,t),
\end{equation}
where the entries of $B(\xi,t)$ are as follows (cf. with (3.32) in \cite{RS}):
\begin{subequations}\label{B}
	\begin{align}
	&B_{11}(\xi,t)=B_{22}(\xi,t)=0,\\
	&B_{12}(\xi,t)=-i\beta(\xi)e^{4it\xi^2+2\sum\limits_{s=0}^{m}\chi_s(-\xi)}
	(8t)^{-i\nu(-\xi)}\prod\limits_{s=0}^{m-1}(\omega_{n-s}-\xi)^{-2},\\
	&B_{21}(\xi,t)=i\gamma(\xi)e^{-4it\xi^2-2\sum\limits_{s=0}^{m}\chi_s(-\xi)}
	(8t)^{i\nu(-\xi)}\prod\limits_{s=0}^{m-1}(\omega_{n-s}-\xi)^{2},
	\end{align}
\end{subequations}
and the remainder is (cf.  (3.33) in \cite{RS}):
\begin{equation}
\tilde r(\xi,t)=
\begin{pmatrix}
O\left(t^{-1-\Im\nu(-\xi)}\right)& O\left(t^{-1+\Im\nu(-\xi)}\right)\\
O\left(t^{-1-\Im\nu(-\xi)}\right)& O\left(t^{-1+\Im\nu(-\xi)}\right)
\end{pmatrix},\quad t\to\infty.
\end{equation}
Further, we evaluate asymptotics of  $\breve M^{as}(x,t,k)$ as $t\to\infty$ using its integral representation in terms of the solution of the  singular integral equation:
\begin{equation}\label{M-int-rep}
\breve M^{as}(x,t,k) = I+\frac{1}{2\pi i}\int_{\hat\Gamma_1}\mu(x,t,s)
(\breve J^{as}(x,t,s)-I)
\frac{ds}{s-k},
\end{equation}
where $\mu$ solves the integral equation $\mu -C_w \mu = I$, with 
$w=\breve J^{as} - I$ and the Cauchy-type operator $C_w$ defined as follows:
\[
C_w f=(C_-f)(k)=\frac{1}{2\pi i}\lim_{
	\substack{k'\to k \\  k'\in -side}}\int_{\hat\Gamma_1}\frac{f(s)}{s-k'}ds.
\]
Since $V(k)$ is uniformly bounded on $\hat\Gamma_1$ and does not depend on $t$ and $x$, we can proceed as in \cite{RS} and conclude that the main term in the large-$t$ evaluation 
of $\breve M^{as}$ in (\ref{M-int-rep}) is given by the integral along the  circle
$S_{-\xi}$.  In this way we obtain the following representation for 
$\breve M^{as}(x,t,k)$ (see (3.30) and (3.34) in \cite{RS}):
\begin{align}\label{M^R}
\nonumber
\lim_{k\to\infty} k\left(\breve M^{as}(x,t,k) - I\right)&=
-\frac{1}{2\pi i}\int_{S_{-\xi}}V(k)\left((\check m_0^{as})^{-1}(x,t,k) - I\right)V^{-1}(k)\,dk + R(\xi,t).
\end{align}
Taking into account (\ref{m_0}) we conclude that
\begin{equation}\label{M^R1}
\lim_{k\to\infty} k\left(\breve M^{as}(x,t,k) - I\right)=
B^{as}(\xi,t)+R(\xi,t),
\end{equation}
where $R(\xi,t)=
\begin{pmatrix}
R_1(\xi,t)& R_1(\xi,t) + R_2(\xi,t)\\ 
R_1(\xi,t)& R_1(\xi,t) + R_2(\xi,t)
\end{pmatrix}$ and (see (\ref{B}))
\begin{equation}
B^{as}(\xi,t)=\frac{1}{\sqrt{8t}}
\begin{pmatrix}
\frac{c_0^{as}(\xi)}{\xi}B_{21}(\xi,t)&
\frac{(c_0^{as}(\xi))^2}{\xi^2}B_{21}(\xi,t)-B_{12}(\xi,t)\\
-B_{21}(\xi,t)& -\frac{c_0^{as}(\xi)}{\xi}B_{21}(\xi,t)
\end{pmatrix}.
\end{equation}
\begin{remark}\label{remRHas}
In the analysis of the Riemann-Hilbert problem (\ref{RHas2}),
 the reflection coefficients $r_j^{as}(k)$, $j=1,2$ (see (\ref{ras})) have the form
$$
r^{as}_1(k)=r_1(k)d^{-2}(k)
\prod\limits_{s=0}^{m-1}\left(\frac{k-p_{n-s}}{k+\omega_{n-s}}\right)^2,
\quad
r^{as}_2(k)=r_2(k)d^2(k)
\prod\limits_{s=0}^{m-1}\left(\frac{k+\omega_{n-s}}{k-p_{n-s}}\right)^2,
$$ where $d(k)=\frac{k}{k-p_{n-m}}$. 
Moreover,
$
V(k)=\begin{pmatrix}1& 0\\-\frac{c_0^{as\#}(\xi)}{k}& 1 \end{pmatrix}
$
in the definition of $\breve M^{as}(x,t,k)$ (see (\ref{breveM})).
Therefore,
\begin{align}
&q(x,t)=2i\lim_{k\to\infty}k\breve M^{as}_{12}(x,t,k),\\
&q(-x,t)=-2i\left(\overline{c_0^{as\#}(\xi)}+\lim_{k\to\infty}k\overline{\breve M^{as}_{21}(x,t,k)}\right),
\end{align}
and  $B^{as}(\xi,t)$ and $R(\xi,t)$ in (\ref{M^R1}) are as follows:
\begin{equation}
B^{as}(\xi,t)=\frac{1}{\sqrt{8t}}
\begin{pmatrix}
-\frac{c_0^{as\#}(\xi)}{\xi}B_{12}(\xi,t)&
-B_{12}(\xi,t)\\
\frac{(c_0^{as\#}(\xi))^2}{\xi^2}B_{12}(\xi,t)-B_{21}(\xi,t)& \frac{c_0^{as\#}(\xi)}{\xi}B_{12}(\xi,t)
\end{pmatrix},
\end{equation}
and  $R(\xi,t)=
\begin{pmatrix}
R_1(\xi,t) + R_2(\xi,t)& R_2(\xi,t)\\ 
R_1(\xi,t) + R_2(\xi,t)& R_2(\xi,t)
\end{pmatrix}$.
\end{remark}


\end{document}